\documentclass[12pt]{amsart}
\usepackage{amsmath,amssymb,mathrsfs,hyperref,color}

\usepackage[nameinlink]{cleveref}
\hypersetup{colorlinks={true},linkcolor={blue},citecolor=green}
\crefformat{equation}{(#2#1#3)}

\usepackage{mathtools}
\usepackage[x11names]{xcolor}
\newtagform{blue}{\color{blue}(}{)}

\usepackage{paralist}

\usepackage{subfigure}
\usepackage{float}
\usepackage{times}
\usepackage{tikz}
\usepackage{cases}
\usetikzlibrary{intersections}
\usepackage{xcolor}
\usepackage[latin1]{inputenc}
\usepackage{bbm}
\usetikzlibrary{shadings}
\usetikzlibrary[intersections]
\usetikzlibrary{arrows,shapes}

\makeatletter
\def\setliststart#1{\setcounter{\@listctr}{#1}%
  \addtocounter{\@listctr}{-1}}
\makeatother

\makeatletter
\@addtoreset{figure}{section}
\makeatother

\usepackage{calc}
\newtheorem{theorem}{Theorem}[section]
\newtheorem{lemma}[theorem]{Lemma}
\newtheorem{proposition}[theorem]{Proposition}
\newtheorem{corollary}[theorem]{Corollary}
\newtheorem{remarks}[theorem]{Remark}
\newtheorem{example}[theorem]{Example}
\newtheorem{definition}[theorem]{Definition}
\numberwithin{equation}{section}

\newcommand{\R}{\mathbb{R}}

\newcommand{\N}{\mathbb{N}}

\newcommand{\K}{\mathcal{K}}

\newcommand{\LL}{\mathcal{L}}

\newcommand{\la}{\Lambda}

\DeclareMathOperator*{\argmin}{argmin} 
\DeclareMathOperator*{\supp}{spt}
\DeclareMathOperator*{\esssup}{ess\ sup}

\DeclareMathOperator*{\eps}{\varepsilon}

\DeclareMathOperator*{\mane}{\alpha({\it L})}
\DeclareMathOperator*{\SR}{SR}

\makeatletter
\def\moverlay{\mathpalette\mov@rlay}
\def\mov@rlay#1#2{\leavevmode\vtop{%
   \baselineskip\z@skip \lineskiplimit-\maxdimen
   \ialign{\hfil$\m@th#1##$\hfil\cr#2\crcr}}}
\newcommand{\charfusion}[3][\mathord]{
    #1{\ifx#1\mathop\vphantom{#2}\fi
        \mathpalette\mov@rlay{#2\cr#3}
      }
    \ifx#1\mathop\expandafter\displaylimits\fi}
\makeatother

\title[Asymptotic analysis for H-J equations]{Asymptotic analysis for Hamilton-Jacobi-Bellman equations \\ on Euclidean space}
\author{Piermarco Cannarsa \and Cristian Mendico}
\address{Dipartimento di Matematica, Universit\`a di Roma ``Tor Vergata'', Via della Ricerca Scientifica 1, 00133 Roma, Italy}
\email{cannarsa@mat.uniroma2.it}
\address{GSSI-Gran Sasso Science Institute, Viale F. Crispi 7, 67100  L'Aquila and CEREMADE, Universit\'e Paris-Dauphine, Place du Mar\'echal de Lattre de Tassigny - 75775 PARIS Cedex 16}
\email{cristian.mendico@gssi.it}

\date{\today}

\begin{document}
\usetagform{blue}
\maketitle
\begin{abstract}
The long-time average behavior of the value function in the calculus of variations is known to be connected to the existence of the limit of the corresponding Abel means. Still in the Tonelli case, such a limit is in turn related to the existence of solutions of the critical Hamilton-Jacobi equation. The goal of this paper is to address similar issues when set on the whole Euclidean space and the Hamiltonian fails to be Tonelli. We first study the convergence of the time-averaged value function as the time horizon goes to infinity, proving the existence of the critical constant for a general control system. Then, we show that the ergodic equation admits solutions for systems associated with a family of vector fields which satisfies the Lie Algebra rank condition. Finally,  we construct a critical solution of the HJB equation on the whole space which coincides with its Lax-Oleinik evolution.
 \end{abstract}

\tableofcontents
\section{Introduction}

The study of the long-time behaviour of solutions to Hamilton-Jacobi equations has a long history going back to the seminal paper \cite{bib:LPV}. Among the many works that have been published on the subject, when the state space is compact and the Hamiltonian is Tonelli, we refer, for instance, to \cite{bib:FAS, bib:FA} and references therein. Moreover, the papers \cite{bib:BJMR},  \cite{bib:BOTT},  \cite{bib:FM},  \cite{bib:HI},  \cite{bib:III},  \cite{bib:ISI} studied the case of a non-compact state space and a Tonelli or quasi-Tonelli Hamiltonian. However, when the Hamiltonian is neither strictly convex nor coercive---as is the case of Hamilton-Jacobi-Bellman equations---such a problem is widely open, even thought specific models were addressed in \cite{bib:CNS},  \cite{bib:XY},  \cite{bib:PCE},  \cite{bib:BA}, and \cite{bib:ALE} where the ergodic problem associated with the so-called $G$-equation or other non-coercive Hamiltonians on compact manifolds are treated. See also \cite{bib:ABG, bib:ABE}  \cite{bib:KDV}, and \cite{bib:MG} for more on second order differential games.

In this paper, we contribute to the aforementioned analysis by studying the long-time behavior of solutions to Hamilton-Jacobi-Bellman equations. As is well-known, for such problems the Hamiltonian fails to be coercive and/or strictly convex.

The first part of this work is devoted to the asymptotic behavior of the value function associated with a general state equation of the form
\begin{equation}\label{eq:generalcontrol}
\dot\gamma(t)=f(\gamma(t), u(t)), \quad (t \geq 0)
\end{equation}
where $u:[0,\infty) \to \R$ is a measurable function and $f: \R^{d} \times \R^{m} \to \R$ is a Lipschitz continuous function (in space) with sublinear growth in space and control. Given a Lagrangian $L: \R^{d} \times \R^{m} \to \R$, an initial position $x \in \R^{d}$, and a time horizon $T > 0$  we consider the problem of minimizing the functional
\begin{equation*}
u \longmapsto \int_{0}^{T}{L(\gamma_{u}^{x}(s), u(s))\ ds}	
 \end{equation*}
over the space of all measurable controls $u:[0,T]\to\R^{m}$, where $\gamma_{u}^{x}$ denotes the solution of \eqref{eq:generalcontrol} such that $\gamma(0)=x$. So, we address the existence of the limit $V_{T}(x)/T$ as $T \to \infty$ where
  \begin{equation*}
   	V_{T}(x)=\inf_{u} \int_{0}^{T}{L(\gamma_{u}^{x}(t), u(t))\ dt} \quad (x \in \R^{d}).
   \end{equation*}
  Our analysis shows that such a limit exists locally uniformly and is independent of the initial position $x \in \R^{d}$, that is, there exists a real number $\mane$ such that for any $R \geq 0$
   \begin{equation}\label{eq:prelimit}
\lim_{T \to \infty} \sup_{x \in \overline{B}_{R}} \left|\frac{1}{T}V_{T}(x) - \mane\right|=0. 
\end{equation} 
For the proof of the above result, a major difficulty is the lack of compactness of the configuration space. In order to cope with such an issue we appeal to condition {\bf (L3)} below, which ensures the existence of a compact attractor for all minimizing trajectories. In particular, such an assumption allows us to show that the excursion time of minimizers for $V_{T}$ from the attractor is locally bounded uniformly w.r.t. $T$, see \Cref{prop:constraint} below. We observe that an hypothesis of the same type as {\bf (L3)} was used, in \cite{bib:CCMW}, to study the long-time behavior of  first-order Mean Field Games systems on Euclidean space  and, in \cite{bib:ISI}, to investigate the limit behavior of  discounted Hamilton-Jacobi equations on the whole space.

Following \cite{bib:FA},  the existence of the limit in  \eqref{eq:prelimit} is known to be related to the existence of a solution $(c,\chi)$ of the ergodic Hamilton-Jacobi equation
\begin{equation}\label{eq:criticalprel}
c +H(x, D\chi(x))=0 \quad (x \in \R^{d})	
\end{equation}
provided that 
\begin{equation*}
H(x,p):=\sup_{u \in \R^{m}}\left\{\langle p, f(x,u) \rangle - L(x,u)\right\}
\end{equation*}
is of Tonelli type.
In the second part of the paper, having the critical constant $\mane$ at our disposal, we study the existence of solutions to \eqref{eq:criticalprel} with $c=\mane$ for a class of Hamiltonians that are not Tonelli, that is, when $f: \R^{d} \times \R^{m} \to \R^{d}$ has the form
\begin{equation*}
f(x,u)=\sum_{i=1}^{m}{u_{i}f_{i}(x)}
\end{equation*} 
where $f_{i}$ are $m \in \{1, \cdots, d\}$ smooth vector fields defined on $\R^{d}$, with sublinear growth.  The main assumption on such a model is the so-called Chow condition (also known as H\"ordmander condition in PDE), i.e., the fact that iterated Lie brackets of $f_{1}$, $\dots, f_{m}$ generate the whole tangent space at any point. Indeed, this condition implies that the system is controllable, that is, given any two points in the state space one can find a control that generates a path which joins the two points. Such systems are naturally associated with a new metric on the state space---the sub-Riemannian metric---which in general fails to be equivalent to the classical Euclidean metric, see for instance \cite{bib:ABB}, \cite{bib:COR}, \cite{bib:LR}, and \cite{bib:MON}. 

By investigating the limit behavior of the discounted Hamilton-Jacobi equation associated with \eqref{eq:criticalprel}, we deduce that the ergodic equation admits solutions for $c=\mane$ (\Cref{thm:convergence}). Then we construct a specific solution of such an equation which coincides with its Lax-Oleinik evolution. Our interest in such a result is motivated by the fact that we need it to derive a further characterization of the ergodic constant as the minimum of the Lagrangian action on closed measures. As we will show in a forthcoming paper \cite{bib:AUB}, this is a crucial step  to investigate the related Aubry set, on which ergodic solutions have important regularity properties.

We want to point out that the results in the second part of this paper are specific to affine-control systems without drift. Indeed, in the presence of a drift, the existence of a continuous viscosity solution to \eqref{eq:criticalprel} with a noncoercive Hamiltonian remains a challenging problem due to the lack of small time local controllability. For instance, for systems with controlled acceleration, no continuous viscosity solution to the associated ergodic equation has been constructed to this date, even thought the existence of the ergodic constant has been proved (see \cite{bib:PC1}).

We conclude this introduction by comparing our results with the theory developed in \cite{bib:ALE}. The analysis in \cite{bib:ALE} applies to compact manifolds and families of 3-generating vector fields (i.e., a step$-2$ Lie algebra) obtaining the existence of the critical constant and ergodic solution. An ergodic solution is also obtained as a fixed point of the Lax-Oleinik semigroup. Our paper differs from \cite{bib:ALE} in the following ways:
\begin{enumerate}
\item all the results are obtained on a non-compact state space, instead of a compact manifold as in \cite{bib:ALE};
\item the existence of the critical constant is proved for general control systems satisfying a certain controllability assumption, see {\bf (GULC)}, which includes both drift-less sub-Riemannian systems and linear control systems which are not 3-generating;
\item our analysis of the Lax-Oleinik semigroup is developed on spaces of unbounded continuous functions as is natural when working on unbounded configuration space;
\item our technic, based on a dynamical approach, is completely different from the method of \cite{bib:ALE}, which uses properties of the optimal cost requiring compactness of the state space. 
\end{enumerate}



\medskip\medskip

\noindent The paper is organized as follows. In \Cref{sec:preliminaries}, we fix the notation used throughout the paper and discuss some preliminaries. In \Cref{sec:Settingsassumptions}, we introduce the model we are interested in and the assumptions that will be in force throughout the paper. In \Cref{sec:compactness}, we derive bounds for optimal trajectories and optimal controls. \Cref{sec:ergodicconstant} contains  the first main result of this paper: we construct the critical constant $\mane$. In \Cref{sec:Abel}, we prove the existence of a corrector, i.e., a viscosity solution of the ergodic Hamilton-Jacobi equation associated with a family of vector fields which satisfies the Lie Algebra rank condition. In \Cref{sec:LaxOleinik}, we prove the existence of a  corrector which is also a fixed-point of the corresponding Lax-Oleinik evolution. In \Cref{sec:Appendix1}, we give a new formulation of the Abelian-Tauberian Theorem, adapted to the settings of this paper.

\medskip\medskip
{\bf Acknowledgement.} The first author was partly supported by Istituto Nazionale di Alta Matematica (GNAMPA 2020 Research Projects) and by the MIUR Excellence Department Project awarded to the Department of Mathematics, University of Rome Tor Vergata, CUP E83C18000100006. The second author was partly supported by Istituto Nazionale di Alta Matematica (GNAMPA 2020 Research Projects) and University Italo Francese (Vinci Project 2020) n. C2-794. Part of this paper was completed while the second author was visiting in Department of Mathematics of the University of Rome Tor Vergata. The authors would like to express their gratitude to the anonymous reviewers for their careful reading of the manuscript and many insightful comments and suggestions.

\medskip\medskip
{\bf Declarations of interests:} None. 

\section{Preliminaries}
\label{sec:preliminaries}
\subsection{Notation}
We write below a list of symbols used throughout this paper.
\begin{itemize}
	\item Denote by $\mathbb{N}$ the set of positive integers, by $\mathbb{R}^d$ the $d$-dimensional real Euclidean space,  by $\langle\cdot,\cdot\rangle$ the Euclidean scalar product, by $|\cdot|$ the usual norm in $\mathbb{R}^d$, and by $B_{R}$ the open ball with center $0$ and radius $R$.
	\item If $\Lambda$ is a real $d\times d$ matrix, we define the norm of $\Lambda$ by 
\[
\|\Lambda\|=\sup_{|x|=1,\ x\in\mathbb{R}^d}|\Lambda x|.
\]

\item For a Lebesgue-measurable subset $A$ of $\mathbb{R}^d$, we let $\mathcal{L}^{d}(A)$ be the $d$-dimensional Lebesgue measure of $A$ and  $\mathbf{1}_{A}:\mathbb{R}^n\rightarrow \{0,1\}$ be the characteristic function of $A$, i.e.,
\begin{align*}
\mathbf{1}_{A}(x)=
\begin{cases}
1  \ \ \ &x\in A,\\
0 &x \not\in A.
\end{cases}
\end{align*} 
We denote by $L^p(A)$ (for $1\leq p\leq \infty$) the space of Lebesgue-measurable functions $f$ with $\|f\|_{p,A}<\infty$, where   
\begin{align*}
& \|f\|_{\infty, A}:=\esssup_{x \in A} |f(x)|,
\\& \|f\|_{p,A}:=\left(\int_{A}|f|^{p}\ dx\right)^{\frac{1}{p}}, \quad 1\leq p<\infty.
\end{align*}
For brevity, $\|f\|_{\infty}$  and $\|f\|_{p}$ stand for  $\|f\|_{\infty,\mathbb{R}^d}$ and  $\|f\|_{p,\mathbb{R}^d}$ respectively.

\item $C_b(\mathbb{R}^d)$ stands for the function space of bounded uniformly  continuous functions on $\mathbb{R}^d$. $C^{2}_{b}(\mathbb{R}^{d})$ stands for the space of bounded functions on $\mathbb{R}^d$ with bounded uniformly continuous first and second derivatives. 
$C^k(\mathbb{R}^{d})$ ($k\in\mathbb{N}$) stands for the function space of $k$-times continuously differentiable functions on $\mathbb{R}^d$, and $C^\infty(\mathbb{R}^{d}):=\cap_{k=0}^\infty C^k(\mathbb{R}^{d})$. 
 $C_c^\infty(\mathbb{R}^{d})$ stands for the space of functions in $C^\infty(\mathbb{R}^{d})$ with compact support. Let $a<b\in\mathbb{R}$.
  $AC([a,b];\mathbb{R}^d)$ denotes the space of absolutely continuous maps $[a,b]\to \mathbb{R}^d$.
  
  \item For $f \in C^{1}(\mathbb{R}^{d})$, the gradient of $f$ is denoted by $Df=(D_{x_{1}}f, ..., D_{x_{d}}f)$, where $D_{x_{i}}f=\frac{\partial f}{\partial x_{i}}$, $i=1,2,\cdots,d$.
Let $k$ be a nonnegative integer and let $\alpha=(\alpha_1,\cdots,\alpha_d)$ be a multiindex of order $k$, i.e., $k=|\alpha|=\alpha_1+\cdots +\alpha_d$ , where each component $\alpha_i$ is a nonnegative integer.   For $f \in C^{k}(\mathbb{R}^{d})$,
define $D^{\alpha}f:= D_{x_{1}}^{\alpha_{1}} \cdot\cdot\cdot D^{\alpha_{d}}_{x_{d}}f$.
\item Let $\eps > 0$ and let $\xi_{\eps} \in C^{\infty}(\R^{d})$ be a smooth mollifier, that is, 
\begin{equation*}
\supp(\xi^{\eps}) \subset B_{\eps}, \quad \xi^{\eps}(x)\geq 0 \,\, \forall\ x \in \R^{d}, \quad \int_{B_{\eps}}{\xi^{\eps}(x)\ dx} = 1.	
\end{equation*}
\end{itemize}

\subsection{Control of nonholonomic systems}
We recall the notion of nonholonomic control systems and sub-Riemannian distance on $\R^{d}$, see \cite{bib:MON, bib:LR, bib:ABB}.

A class of nonholonomic drift-less systems on $\R^{d}$ is a control system of the form
\begin{equation}\label{eq:preliminary}
	\dot\gamma(t)=\sum_{i=1}^{m}{f_{i}(\gamma(t))u_{i}(t)}, \quad t \in [0,+\infty)
\end{equation}
Such a system induces a distance on $\R^{d}$ in the following way. First, we define the sub-Riemannian metric to be the function $g: \R^{d} \times \R^{m} \to \R \cup \{ \infty\}$ given by
\begin{equation*}
g(x,v)=\inf\left\{\sum_{i =1}^{m}{u_{i}^{2}: v=\sum_{i=1}^{m}{f_{i}(x)u_{i}}}\right\}.	
\end{equation*}
If $v \in \text{span}\{f_{1}(x), \dots, f_{m}(x)\}$ then the infimum is attained at a unique value $u_{x} \in \R^{m}$ and $g(x,v)=| u_{x}|^{2}$. Then, since $g(\gamma(t), \dot\gamma(t))$ is measurable, being the composition of the lower semicontinuous function $g$ with a measurable function, we can define the length of an absolutely continuous curve $\gamma:[0,1] \to \R^{d}$ as
\begin{equation*}
\text{length}(\gamma)=\int_{0}^{1}{\sqrt{g(\gamma(t), \dot\gamma(t))}\ dt}.	
\end{equation*}
 In conclusion, one defines the sub-Riemannian distance as
 \begin{equation*}
 d_{\text{SR}}(x,y)=\inf_{(\gamma, u) \in \Gamma_{0,1}^{x \to y}}\text{length}(\gamma)	
 \end{equation*}
where $\Gamma_{0,1}^{x \to y}$ denotes the set of all trajectory-control pairs such that $u \in L^{2}(0,1; \R^{m})$, $\gamma$ solves \eqref{eq:preldyn} for such a control $u$, $\gamma(0)=x$ and $\gamma(1)=y$. Following \cite{bib:CR} it is possible to represent the sub-Riemannian distance as follows
\begin{equation}\label{eq:subriem}
	d_{\SR}(x,y) = \inf\left\{T >0: \exists\ (\gamma, u) \in \Gamma_{0,T}^{x \to y},\ |u(t)| \leq 1\ \text{a.e.}\ t \in [0,T] \right\}
\end{equation}
for any $x$, $y \in \R^{d}$. Moreover, again from \cite{bib:CR} the sub-Riemannian distance can be characterised in terms of the sub-Riemannian energy: setting
\begin{equation*}
e_{\SR}(x,y)=\inf_{(\gamma,u) \in \Gamma_{0,1}^{x\to y}} \int_{0}^{1}{g(\gamma(t), \dot\gamma(t))\ dt},
\end{equation*}
one can prove that 
\begin{equation}\label{eq:energy}
	d_{\SR}(x,y)=\sqrt{e_{\SR}(x,y)}
\end{equation}
(see, for instance, \cite[Lemma 11]{bib:CR}). 
\begin{remarks}\em
Note that since $g$ is homogenous of degree 1, the sub-Riemannian distance and the sub-Riemannian energy function can be equivalently defined by
\begin{equation*}
d_{\SR}(x,y)=\inf_{(\gamma, u) \in \Gamma_{0,t}^{x \to y}}\text{length}(\gamma)
\end{equation*}
and
\begin{equation*}
e_{\SR}(x,y)=\inf_{(\gamma,u) \in \Gamma_{0,t}^{x\to y}} \int_{0}^{t}{g(\gamma(s), \dot\gamma(s))\ ds},
\end{equation*}
for any $t \geq 0$. 
\end{remarks}

In this paper, we are interested in the controllability properties of system \eqref{eq:preliminary}. For such a system, controllability can be obtained by using the Lie algebra generated by $f_{1}$, $\dots$, $f_{m}$, which is defined as follows. Set
 \begin{equation}\label{eq:delta1}
 \Delta^{1} = \text{span}\{f_{1}, \dots, f_{m}\}	
 \end{equation}
and, for any integer $s \geq 1$,
\begin{equation}\label{eq:delta2}
\Delta^{s+1}= \Delta^{s} + [\Delta^{1}, \Delta^{s}]	
\end{equation}
where $[\Delta^{1}, \Delta^{s}]:=\text{span}\{[X, Y]: X \in \Delta^{1}, Y \in \Delta^{s}\}$. The  Lie algebra generated by $f_{1},\dots,f_{m}$ is defined as
\begin{equation*}
\text{Lie}(f_{1}, \dots, f_{m})=\bigcup_{s \geq 1}\Delta^{s}.	
\end{equation*}
We say that system \eqref{eq:preliminary} satisfies {\it Chow's condition} if $\text{Lie}(f_{1}, \dots, f_{m})(x)=\R^{d}$ for any $x \in \R^{d}$, where $\text{Lie}(f_{1}, \dots, f_{m})(x) = \{X(x): X \in \text{Lie}(f_{1}, \dots, f_{m})\}$. Equivalently, for any $x \in \R^{d}$ there exists an integer $r \geq 1$ such that $\Delta^{r} (x)=\R^{d}$. The minimum integer with such a property is called the degree of nonholonomy at $x$ and will be denoted by $r(x)$. Chow's condition is also known as the {\it Lie algebra rank condition} (LARC) in control theory and as the {\it H\"ormander condition} in the context of PDEs. 
\begin{example}\label{ex:examples}\em
The following are two well-known examples of sub-Riemannian systems for which Chow's condition holds true.
	\begin{itemize}
		\item[($i$)] {\bf Heisenberg group:} We consider the system in $\R^{3}$
\begin{align*}
\begin{cases}
	\dot{x}(t) &=u(t),
	\\
	\dot{y}(t) &=v(t),
	\\
	\dot{z}(t) &=u(t)y(t)-v(t)x(t) 
\end{cases}	
\end{align*}
In this case, the matrix of the system is given by
\begin{equation*}
A(x,y,z)=
\begin{bmatrix}
	1 & 0 \\
	0 & 1 \\
	y & -x
\end{bmatrix}
\end{equation*}	
and the columns of such matrix satisfy the H\"ordmander condition: $X_{1} = (1,0,y)$, $X_{2}=(0,1,-x)$ and $[X_{1}, X_{2}]=(0,0,2)$ generate $\R^{3}$. 
		\item[($ii$)] {\bf Grushin type systems:} Consider a control system of the form
\begin{align*}
\begin{cases}
	\dot{x}(t) &=u(t),
	\\
	\dot{y}(t) &=\varphi(x(t))v(t)
\end{cases}	
\end{align*}
for a nonzero continuous function $\varphi(x)$ with sub-linear growth. The classical Grushin system in $\R^{2}$ is obtained taking $\varphi(x)=x$. Then, the dynamics is given by the matrix
\begin{equation*}
A(x,y)=
\begin{bmatrix}
	1 & 0 \\
	0 & x
\end{bmatrix}
\end{equation*}	
		whose columns satisfy the H\"ordmander condition: $X_{1}=(1,0)$ and $[X_{1}, X_{2}]=(0,1)$ generates $\R^{2}$. \qed
	\end{itemize}
\end{example}

\begin{theorem}[{\bf Chow-Rashevsky theorem, \cite[Theorem 3.1.8]{bib:COR}}]
	If system \eqref{eq:preliminary} satisfies Chow's condition, then any two points in $\R^{d}$ can be joined by a trajectory satisfying \eqref{eq:preliminary}. 
\end{theorem}

Besides controllability, another important consequence of Chow's condition is the well-known Ball-Box Theorem, see for instance \cite[Theorem 10.67]{bib:ABB}. Of particular interest to us is a corollary of such a theorem which gives  H\"older equivalence between the Euclidean distance and the sub-Riemannian one. 
First, we observe that for any $x \in \R^{d}$ a continuity argument ensures the existence of a neighborhood $U_{x}$ of $x$ such that 
\begin{equation}\label{eq:compactdegree}
	\Delta^{r(x)}(y)=\R^{d}, \quad \forall\ y \in U_{x}.
\end{equation}
Thus, given a compact set $\K$ there exists a finite cover given by $\{U_{x_{i}}\}_{i=1, \dots, N}$ and a set of integers $\{r(x_{i})\}_{i=1, \dots, N}$ such that \eqref{eq:compactdegree} holds on $U_{x_{i}}$ with $r(x)=r(x_{i})$. Taking $$r=\displaystyle{\max_{i=1, \dots, N}} r(x_{i})$$ we obtain 
\begin{equation}\label{eq:comp1}
\Delta^{r}(y)=\R^{d}, \quad \forall\ y \in \K.	
\end{equation}
We call degree of nonholonomy of $\K$ the minimum integer such that \eqref{eq:comp1} holds true and we denote it by $r(\K)$. Moreover, we recall that a family of vector fields $\{f_{i}\}_{i=1, \dots, m}$ is an equi-regular distribution on $\R^{d}$ if there exists $r_{0} \geq 1$ such that $\Delta^{r_{0}}(x)=\R^{d}$ for any $x \in \R^{d}$.

\begin{corollary}\label{cor:ballbox}
	For any compact set $\mathcal{K} \subset \R^{d}$ there exist two constants $\tilde{c}_{1}$, $\tilde{c}_{2} > 0$ such that
	\begin{equation}\label{eq:distancequivalence}
\widetilde{c}_{1}|x-y| \leq d_{\SR}(x,y) \leq \widetilde{c}_{2}|x-y|^\frac{1}{r(\mathcal{K})}, \quad \forall\ x,y \in \mathcal{K}.	
\end{equation} 
\end{corollary}

Furthermore, we recall that the topology induced by $(\R^{d}, d_{\SR})$ coincides with the topology induced by the Euclidean distance on $\R^{d}$ (\cite[Theorem 3.31]{bib:ABB}). In particular, from this result, we obtain that a set is compact in $(\R^{d}, d_{\SR})$ if and only if it is compact in $\R^{d}$ w.r.t. Euclidean distance.

\section{Settings and assumptions}
\label{sec:Settingsassumptions}

Let $f: \R^{d} \times \R^{m} \to \R$ be a continuous function satisfying the following. 
\begin{itemize}
\item[{\bf (F0)}] There exists a constant $c_{f} \geq 0$ such that 
\begin{equation*}
|f(x,u) - f(y,u)| \leq c_{f}(1+|u|)|x-y|, \quad  \forall\ (x,u) \in \R^{d} \times \R^{m}
\end{equation*}
and
\begin{equation*}
|f(x,u)| \leq c_{f}(1+|u|)(1+|x|), \quad \forall\ (x,u) \in \R^{d} \times \R^{m}. 
\end{equation*}
\end{itemize}
We consider the control system 
\begin{equation}\label{eq:dynamics}
\dot\gamma(t)=f(\gamma(t), u(t)), \quad t \in [0, \infty). 
\end{equation}
For any $s_{0}$, $s_{1} \in \R$ such that $s_{0} < s_{1}$ and $x$, $y \in \R^{d}$ we set
	\begin{align*}\label{eq:controlnotation}
	\begin{split}
	\Gamma_{s_{0}, s_{1}}^{x \to} & =\{(\gamma, u) \in \text{AC}([s_{0}, s_{1}]; \R^{d}) \times L^{2}(s_{0}, s_{1}; \R^{m}): \dot\gamma(t)=f(\gamma(t), u(t)),\,\, \gamma(s_{0})=x\},
	\\
	 \Gamma_{s_{0}, s_{1}}^{\to y} & =\{(\gamma, u) \in \text{AC}([s_{0}, s_{1}]; \R^{d}) \times L^{2}(s_{0}, s_{1}; \R^{m}): \dot\gamma(t)=f(\gamma(t), u(t)),\,\, \gamma(s_{1})=y\},
	 \\
	 \Gamma_{s_{0}, s_{1}}^{x \to y} & = \Gamma_{s_{0}, s_{1}}^{x \to} \cap \Gamma_{s_{0}, s_{1}}^{\to y}.
	 \end{split}
	\end{align*}
%

	By {\bf (F0)} and Gronwall's inequality we  get the following estimate for the trajectories of \eqref{eq:dynamics}. 
\begin{lemma}\label{lem:boundedtrajectories}
Let $x \in \R^{d}$, $t \geq 0$ and $(\gamma, u) \in \Gamma_{0,t}^{x \to}$. If $u \in L^{\infty}(0,t; \R^{m})$ then we have that
	\begin{equation*}
	|\gamma(s)| \leq (|x| + c_{f}(1 + \|u\|_{\infty})s)e^{c_{f}(1 + \|u\|_{\infty})s}, \quad \forall\ s \in [0,t].
	\end{equation*}
	\end{lemma}

Let the Lagrangian $L: \R^{d} \times \R^{m} \to \R$ be a $C^{2}(\R^{d} \times \R^{m}; \R)$ function satisfying the following. 
\begin{itemize}
\item[{\bf (L1)}] There exist a non-decreasing function $\beta: [0,\infty) \to \R$ and a constant  $\ell_{1} \geq 0$ such that
\begin{align*}
L(x,u) & \leq\ \beta(|x|)(1+ |u|^{2}), \quad \forall\ (x, u) \in \R^{d} \times \R^{m}
\\
|D_{x}L(x,u)| & \leq\ \ell_{1}(1+|u|^{2}), \quad \forall\ (x, u) \in \R^{d} \times \R^{m}
\\
D^{2}_{u}L(x,u) & \geq\ \frac{1}{\ell_{1}}, \quad \forall\ (x, u) \in \R^{d} \times \R^{m}.
\end{align*}
\item[({\bf L2})] There exists $u^{*} \in \R^{m}$ such that $L(x,u) \geq L(x,u^{*})$ for any $(x,u) \in \R^{d} \times \R^{m}$.
\item[({\bf L3})] There exists a compact set $\mathcal{K} \subset \R^{d}$ and a constant $\Theta > 0 $ such that 
	\begin{equation}\label{eq:gap}
		 \inf_{x \in \R^{d} \backslash \mathcal{K}} L(x,u) \geq \Theta  + \min_{x \in \mathcal{K}} L(x,u^{*})
	\end{equation} 
	and there exists $x^{*} \in \argmin_{x \in \K}L(x, u^{*})$ such that
	\begin{equation*}
	f(x^{*}, u^{*})=0.
	\end{equation*}
\end{itemize}

	Note that from the previous assumption on $L$ we obtain 
\begin{equation}\label{eq:L0}
	L(x,u) \geq \frac{1}{2\ell_{1}}|u-u^{*}|^{2} + L(x^{*}, u^{*}), \quad \forall\ (x,u) \in \R^{d} \times \R^{m}.
\end{equation}
Let $H: \R^{d} \times \R^{d} \to \R$ be the Hamiltonian associated with $L$, defined by 
\begin{equation}\label{eq:Hamiltonian}
H(x,p)=\sup_{u \in \R^{m}} \big\{\langle p, f(x,u) \rangle - L(x,u) \big\}.
\end{equation}

\begin{definition}
We say that system \eqref{eq:dynamics} is  {\em locally uniformly globally controllable}, a property denoted by {\bf (LUGC)}, if for any $R \geq 0$ there exist $T_{R} \geq 0$ and $C_{R} \geq 0$ such that for any $x$, $y \in \overline{B}_{R}$ there exists $(\gamma_{x \to y}, u_{x \to y}) \in \Gamma_{0,T_{R}}^{x \to y}$ with
\begin{equation*}
\int_{0}^{T_{R}}{|u_{x \to y}(s)|^{2}\ ds} \leq C_{R}. 
\end{equation*}
\end{definition}

Next, we give some examples of control systems satisfying {\bf (L2)} -- {\bf (L3)} and {\bf (LUGC)}.  

\begin{example}[{\bf Linear control case}]\label{ex:exam}\em
Let $\gamma :[0, \infty) \to \R^{d}$,  let $u: [0,\infty) \to \R^{m}$ and let $A$, $B$ be real $n \times n$ and $n \times m$ matrices, respectively. Then, a linear control system has the form 
\begin{equation*}
\dot\gamma(t) = A \gamma(t) + B u(t), \quad t \geq 0. 
\end{equation*}
Assume that the controllability matrix $Q_{T}$, defined by
\begin{equation*}
Q_{T}=\int_{0}^{T}{e^{tA}BB^{\star} e^{tA^{\star}}\ dt}
\end{equation*}
 is positive definite for some (thus, for all) $T > 0$. 
Then, it is well-known that such system is exactly controllable and from \cite[Proposition 1.1]{bib:ZZ} we get that {\bf (LUGC)} is satisfied. The following are specific cases of linear control system satisfying {\bf (LUGC)} for which we can also give an explicit construction of $(x^{*}, u^{*})$ in {\bf (L2)} -- {\bf (L3)}.

{\bf Control of acceleration.} Consider the linear state equation 
\begin{equation*}
\frac{d}{dt} \left(\begin{matrix}
\gamma(t) \\
\dot\gamma(t)
\end{matrix}\right)
=
\left(\begin{matrix}
0 & 1 \\
0 & 0
\end{matrix}\right) 
\left(\begin{matrix}
\gamma(t) \\
\dot\gamma(t)
\end{matrix}\right)
+ \left(\begin{matrix}
0 \\ 
1
\end{matrix}\right)
 \left(\begin{matrix}
0 \\ 
u(t)
\end{matrix}\right), \quad t \geq 0
\end{equation*}
with $\gamma: [0,+\infty) \to \R^{d}$ and $u: [0,+\infty) \to \R^{d}$.
It is well-known that {\bf (LUGC)} is satisfied, see e.g. \cite[Exercise 1.2]{bib:ZZ}  Now, let us consider a Lagrangian $L$ of the form 
\begin{equation*}
L(x,v,u) = \frac{1}{2}|u|^{2} + \frac{1}{2}|v|^{2} + g(x)
\end{equation*}
where $g: \R^{d} \to \R$ is a bounded and continuous function. Then, one can immediately observe that {\bf (L2)} -- {\bf (L3)} holds true for $u^{*}=0$ and $(x^{*}, v^{*}) \in \R^{d} \times \{0\}$ where $g(x^{*})=\displaystyle{\min_{x \in \R^{d}}} g(x)$.

{\bf Controlled harmonic oscillator.} Consider the system 
\begin{equation*}
\frac{d}{dt} \left(\begin{matrix}
\gamma(t) \\
\dot\gamma(t)
\end{matrix}\right)
=
\left(\begin{matrix}
0 & 1 \\
-1 & 0
\end{matrix}\right) 
\left(\begin{matrix}
\gamma(t) \\
\dot\gamma(t)
\end{matrix}\right)
+ \left(\begin{matrix}
0 \\ 
1
\end{matrix}\right)
 \left(\begin{matrix}
0 \\ 
u(t)
\end{matrix}\right), \quad t \geq 0,
\end{equation*}
which is also locally uniformly globally controllable as one can easily check either directly or by Kalman's rank condition. 
Fix $x^{*} \in \R^{d}$ and consider a Lagrangian $L$ of the form 
\begin{equation*}
L(x,y,u)=\frac{1}{2}|u-x^{*}|^{2} + \frac{1}{2}|y|^{2} + g(x)
\end{equation*}
with $g: \R^{d} \to \R$ a continuous function such that $g(x) \geq g(x^{*})$ for any $ x \in \R^{d}$. Then, it is easy to see that {\bf (L2)} -- {\bf (L3)} are satisfied for $(x^{*}, y^{*}) \in \R^{d} \times \{0\}$ and $u^{*}=x^{*}$. 
\end{example}
\begin{example}[{\bf Sub-Riemannian systems}]\em
 For $m \in \N$ and $i=1, \dots, m$, let 
	\begin{equation*}
		f_{i}: \R^{d} \to \R^{d} 
	\end{equation*}
	be smooth vector fields satisfying the Chow's condition. Then,  the sub-Riemannian system
	\begin{align*}
         \dot\gamma(t)= \displaystyle{\sum_{i=1}^{m}{f_{i}(\gamma(t))u_{i}(t)}}, \quad t \geq 0,
	\end{align*}
	where $u_{i} : [0,\infty) \to \R$ are measurable functions, satisfies {\bf (LUGC)} due to the characterization  of the sub-Riemannian distance in \eqref{eq:subriem}. Thus, from the linearity of the state equation w.r.t. $u$ it follows that {\bf (L2)} -- {\bf (L3)} are satisfied for $u^{*}=0$ and any $x^{*} \in \R^{d}$ such that $L(x,0) \geq L(x^{*}, 0)$. 
\end{example}


\section{Boundedness of optimal trajectories}
\label{sec:compactness}

We consider the following minimization problem: for any $T > 0$ and $x \in \R^{d}$
 \begin{equation}\label{eq:minimization}
 \text{to minimize}\ \int_{0}^{T}{L(\gamma(s), u(s))\ ds}\,\, \text{over all}\ (\gamma, u) \in \Gamma_{0,T}^{x \to}	
 \end{equation}
and we set
\begin{equation}\label{eq:evoValue}
V_{T}(x) = \inf_{(\gamma, u) \in \Gamma_{0,T}^{x \to}} \int_{0}^{T}{L(\gamma(s), u(s))\ ds}, \quad \forall\ x \in \R^{d}.	
\end{equation}

For any $x \in \R^{d}$ we say that a trajectory-control pair $(\gamma, u) \in \Gamma_{0,T}^{x \to}$ is optimal if it solves \eqref{eq:minimization}. For simplicity of notation, we set
\begin{equation*}
\delta^{*}(x)=T_{|x| \vee |x^{*}|}
\end{equation*}
where $T_{|x| \vee |x^{*}|}$ is given by {\bf (LUGC)}.
\begin{remarks}\label{rem:existenceminimizers}\em
\begin{enumerate}
	\item We observe that, by using classical technics from optimal control one can easily obtain the existence of optimal trajectory-control pairs for \eqref{eq:minimization} (see, for instance, \cite[Theorem 7.4.4]{bib:SC}). 
	\item Note that for reversible Lagrangians we have that the function 
	\begin{equation*}
	V_{t}(x) = \inf_{(\gamma, u) \in \Gamma_{0,t}^{x \to}} \int_{0}^{t}{L(\gamma(s), u(s))\ ds}, \quad \forall\ x \in \R^{d}, \, \forall\ t \geq 0	
	\end{equation*}
	solves the Hamilton-Jacobi equation
	\begin{equation*}
	\begin{cases}
	\partial_{t} V_{t}(x) + H(x, D_{x}V_{t}(x))=0, \quad (t, x) \in [0,+\infty) \times \R^{d}
	\\
	V_{0}(x)=0,
	\end{cases}
	\end{equation*}
	in the viscosity sense, where $H$ is defined in \eqref{eq:Hamiltonian}.
	\end{enumerate}
\end{remarks}
In this section, we prove the uniform boundedness of optimal trajectories for \eqref{eq:minimization} starting from a given compact set. We begin by deriving a uniform bound for the Lebsegue measure of all times at which an optimal trajectory may lie outside the compact set $\K$ of assumption {\bf (L3)}.

\begin{proposition}\label{prop:constraint}
	Assume {\bf (F0)}, {\bf (L1)} -- {\bf (L3)} and {\bf (LUGC)}. For any $R \geq 0$ there exists a constant $M_{R} \geq 0$ such that for any $x \in \overline{B}_{R}$, any $T \geq \delta^{*}(x)$, and any optimal pair $(\gamma_{x}, u_{x}) \in \Gamma_{0,T}^{x \to}$ for \eqref{eq:minimization}  we have that 
	\begin{equation}\label{eq:measure1}
		\LL^{1}\left(\{ t \in [0, T] : \gamma_{x}(t) \not\in \K\} \right) \leq M_{R}.
	\end{equation}
\end{proposition}
\noindent {\it Proof.} Fix $R \geq 0$ and let $x \in \overline{B}_{R}$. Let $(\bar\gamma_{x}, \bar{u}_{x}) \in \Gamma_{0,\delta^{*}(x)}^{x \to x^{*}}$ be given by {\bf (LUGC)}. Define the control 
\begin{align*}
	\widehat{u}_{x}(t)=
	\begin{cases}
		\bar{u}_{x}(t), & \quad t \in [0, \delta^{*}(x)]
		\\
		u^{*}, & \quad t \in (\delta^{*}(x), T].
	\end{cases}
\end{align*}
Then,  $(\widehat\gamma_{x}, \widehat{u}_{x}) \in \Gamma_{0,T}^{x \to x^{*}}$ and we obtain
\begin{align}\label{eq:ineq1}
\begin{split}
V_{T}(x) \leq & \int_{0}^{\delta^{*}(x)}{L(\bar\gamma_{x}(t), \bar{u}_{x}(t))\ dt} + 	(T-\delta^{*}(x))L(x^{*}, u^{*})
\\
=:\ & c_{1}(x,\delta^{*}(x)) + (T-\delta^{*}(x))L(x^{*},u^{*}).
\end{split}
\end{align}
Let us estimate $c_{1}(x, \delta^{*}(x))$. Recall that from {\bf (LUGC)} we have that $\|\bar{u}_{x} \|^{2}_{2, [0,\delta^{*}(x)]} \leq C_{R}$. So, \Cref{lem:boundedtrajectories} ensures that 
\begin{equation}\label{eq:control}
|\bar\gamma_{x}(t)| \leq (R + c_{f}(1+C_{R})\delta^{*}(R))e^{c_{f}(1+C_{R})\delta^{*}(R)} =: \la(R), \quad \forall\ t \in [0,\delta^{*}(x)].
\end{equation}
Therefore, by assumption {\bf (L2)} we deduce that
\begin{equation*}
	\int_{0}^{\delta^{*}(x)}{L(\bar\gamma_{x}(t), \bar{u}_{x}(t))\ dt} \leq \beta(\la(R))(1+C_{R})
\end{equation*}
Now, let $(\gamma_{x}, u_{x}) \in \Gamma_{0,T}^{x \to}$ be optimal for \eqref{eq:minimization}. Then, we have that
\begin{align}\label{eq:ineq2}
\begin{split}
 \quad \quad & V_{T}(x) = \int_{0}^{T}{L(\gamma_{x}(t), u_{x}(t))\ dt} 
 \\
 \geq\ & \  \int_{0}^{T}{L(\gamma_{x}(t), u_{x}(t)){\bf 1}_{\K}(\gamma_{x}(t))\ dt}  
+  \int_{0}^{T}{L(\gamma_{x}(t), u_{x}(t)){\bf 1}_{\K^{c}}(\gamma_{x}(t))\ dt}
 \\
 \geq\ &  L(x^{*}, u^{*}) \mathcal{L}^{1}\big(\{t \in [0,T]:\ \gamma_{x}(t) \in \K\}\big) 
 \\
 +\ & \left(\inf_{x \in \R^{d} \backslash \K} L(x,u^{*})\right)\mathcal{L}^{1}\big( \{t \in [0,T]:\ \gamma_{x}(t) \not\in \K\}\big) 
 \\
=\ &  L(x^{*}, u^{*}) \mathcal{L}^{1}\big(\{t \in [\delta^{*}(x),T]:\ \gamma_{x}(t) \in \K\}\big) 
 \\
 +\ & \left(\inf_{x \in \R^{d} \backslash \K} L(x,u^{*})\right)\left(T -\mathcal{L}^{1}\big( \{t \in [\delta^{*}(x),T]:\ \gamma_{x}(t) \in \K\}\big)\right).
 \end{split}
\end{align}
So, combining \eqref{eq:ineq1} and \eqref{eq:ineq2}  and recalling {\bf (L3)}, we deduce that 
\begin{align*}
& \beta(\la(R))(1+C_{R})\
\\
\geq\ & \Big(\inf_{x \in \R^{d} \backslash \K} L(x,u^{*}) - L(x^{*}, u^{*})\Big)\ \left(T - \mathcal{L}^{1}\big( \{t \in [0,T] :\ \gamma_{{u}}(t) \in \K \} \big) \right) + \delta^{*}(x)L(x^{*}, u^{*})
 \\
 \geq\ & \Theta \mathcal{L}^{1}\big( \{t \in [0,T] :\ \gamma_{x}(t) \not\in \K \}\big) + \delta^{*}(x)L(x^{*}, u^{*}).
\end{align*}
Therefore, 
\begin{equation*}
\mathcal{L}^{1}\big(\{ t \in [0,T]:\ \gamma_{x}(t) \not\in \K \} \big) \leq \frac{\beta(\la(R))(1+C_{R}) - \delta^{*}(x)L(x^{*}, u^{*})}{\Theta} . 
\end{equation*}
Setting 
\begin{equation*}
	M_{R}:=\frac{\beta(\la(R))(1+C_{R})-\delta^{*}(R) L(x^{*}, u^{*})}{\Theta} 
\end{equation*}
we obtain the conclusion. \qed

\begin{theorem}\label{thm:boundednesstraj}
	Assume {\bf (F0)}, {\bf (L1)} -- {\bf (L3)} and {\bf (LUGC)}. For any $R \geq 0$ there exist two constants $P_{R}$, $Q_{R} \geq 0$ such that for any $x \in \overline{B}_{R}$, any $T \geq \delta^{*}(x)$, and any optimal pair $(\gamma_{x}, u_{x}) \in \Gamma_{0,T}^{x \to}$ for \eqref{eq:minimization} we have that
	\begin{equation}\label{eq:l2bounds}
	\int_{0}^{T}{|u_{x}(t)-u^{*}|^{2}\ dt} \leq P_{R}
	\end{equation}
	and
	\begin{equation}\label{eq:compactnesstrajectory}
	|\gamma_{x}(t)| \leq Q_{R}, \quad \forall\ t \in [0,T].	
	\end{equation}	
\end{theorem}

\proof
We begin with the proof of \eqref{eq:l2bounds}. 
 Since $(\gamma_{x}, u_{x}) \in \Gamma_{0,T}^{x \to}$ is optimal for \eqref{eq:minimization} we have that
\begin{align}\label{eq:ee1}
\begin{split}
V_{T}(x) = & \int_{0}^{T}{L(\gamma_{x}(t), u_{x}(t))\ dt} 	
\geq \frac{1}{2\ell_{1}}\int_{0}^{T}{|u_{x}(t)-u^{*}|^{2}\ dt} +TL(x^{*},u^{*}).
\end{split}
\end{align}
On the other hand, let $(\bar\gamma_{x}, \bar{u}_{x}) \in \Gamma_{0,\delta^{*}(x)}^{x \to x^{*}}$ be given by {\bf (LUGC)} and define the control 
\begin{align*}
\widehat{u}_{x}(t)=
\begin{cases}
	\bar{u}_{x}(t), & \quad t \in [0,\delta^{*}(x)]
	\\
	u^{*}, & \quad t \in (\delta^{*}(x), T],
\end{cases}	
\end{align*}
that is, $(\widehat\gamma_{x}, \widehat{u}_{x}) \in \Gamma_{0,T}^{x \to x^{*}}$. 
By the definition of $V_{T}$ we deduce that
\begin{align}\label{eq:ee2}
V_{T}(x) \leq \int_{0}^{\delta^{*}(x)}{L(\bar\gamma_{x}(t), \bar{u}_{x}(t))\ dt} + (T-\delta^{*}(x))L(x^{*},u^{*}).	
\end{align}
Combining \eqref{eq:ee1} and \eqref{eq:ee2} we obtain 
\begin{align*}
\frac{1}{2\ell_{1}}\int_{0}^{T}{|u_{x}(t)-u^{*}|^{2}\ dt} \leq \int_{0}^{\delta^{*}(x)}{L(\bar\gamma_{x}(t), \bar{u}_{x}(t))\ dt} - \delta^{*}(x) L(x^{*},u^{*}).	
\end{align*}
In order to prove \eqref{eq:l2bounds}, we need an upper bound for the term 
\begin{equation*}
	\int_{0}^{\delta^{*}(x)}{L(\bar\gamma_{x}(t), \bar{u}_{x}(t))\ dt}.
\end{equation*}
Recall that from {\bf (LUGC)} we have that $\|\bar{u}_{x} \|^{2}_{2, [0,\delta^{*}(x)]} \leq C_{R}$. So, \Cref{lem:boundedtrajectories} ensures that 
\begin{equation}\label{eq:control}
|\bar\gamma_{x}(t)| \leq (R + c_{f}(1+C_{R})\delta^{*}(R))e^{c_{f}(1+C_{R})\delta^{*}(R)} =: \la(R), \quad \forall\ t \in [0,\delta^{*}(x)].
\end{equation}
Therefore, by assumption {\bf (L2)} we deduce that
\begin{equation*}
	\int_{0}^{\delta^{*}(x)}{L(\bar\gamma_{x}(t), \bar{u}_{x}(t))\ dt} \leq \beta(\la(R))(1+C_{R})
\end{equation*}
Hence, \eqref{eq:l2bounds} follows taking
\begin{equation*}
P_{R} = 2(\beta(\la(R))(1+C_{R}) + \delta^{*}(R) L(x^{*}, u^{*})). 
\end{equation*}

We now proceed to prove \eqref{eq:compactnesstrajectory}. Let $(\gamma_{x}, u_{x}) \in \Gamma_{0,T}^{x \to}$ be a solution of \eqref{eq:minimization}. 
Clearly, we just need to estimate $|\gamma_{x}(t)|$ for all times at which the optimal trajectory lies outside the compact set $\K$. Let $\tau \in [0, T]$ be such that $\gamma_{x}(\tau) \not\in \K$ and set
\begin{equation*}
t_{0} = 
\begin{cases}
	\sup\{t \in [0, \tau\ ]: \gamma_{x}(t) \in \K\}, & \text{if}\,\, \{t \in [0, \tau\ ]: \gamma_{x}(t) \in \K\}\not= \emptyset
	\\
	0, & \text{if}\,\, \{t \in [0, \tau\ ]: \gamma_{x}(t) \in \K\}= \emptyset
\end{cases}
\end{equation*}
Let us first consider the case $t_{0} \not= 0$. Since $\gamma_{x}$ is a solution of \eqref{eq:dynamics}, recalling that $|t-t_{0}| \leq M_{R}$ we deduce that for any $t \in [t_{0},\tau]$
\begin{align*}
|\gamma_{x}(t)| \leq\ & |\gamma_{x}(t_{0})| + c_{f}\int_{t_{0}}^{t}{(1+|u_{x}(s)|)(1+|\gamma_{x}(s)|)\ ds}	
\\
\leq & |\gamma_{x}(t_{0})| +  c_{f}(M_{R}(1 + |u^{*}|) + \| u_{x}-u^{*} \|_{2, [t_{0}, \bar{t}]}) \left(\int_{t_{0}}^{t}{\big(1+|\gamma_{x}(s)|^{2}\big)\ ds} \right)^{\frac{1}{2}}.
\end{align*}
Hence, appealing to \eqref{eq:measure1} and \eqref{eq:l2bounds} we deduce that 
\begin{align*}
	|\gamma_{x}(t)|^{2} \leq\ & \bar{C}\Big(|\gamma_{x}(t_{0})|^{2} + M_{R}^{2}(1 + |u^{*}|^{2}) + \| u_{x} - u^{*}\|_{2,[t_{0},\bar{t}]}^{2}M_{R} 
	\\
	+\ & (M_{R} + \| u_{x}-u^{*}\|_{2,[t_{0},\bar{t}]}^{2}) \int_{t_{0}}^{t}{|\gamma_{x}(s)|^{2}\ ds}  \Big)
	\\
	\leq\ & \bar{C}\left(|\gamma_{x}(t_{0})|^{2} + M_{R}^{2} + |u^{*}|^{2} + P_{R}M_{R} + (M_{R} + P_{R}) \int_{t_{0}}^{t}{|\gamma_{x}(s)|^{2}\ ds}\right)
\end{align*}
for some constant $\bar{C} \geq 0$. Thus, Gronwall's inequality yields 
\begin{equation*}
	|\gamma_{x}(t)| \leq \bar{C}|\gamma_{x}(t_{0})|^{2}(M_{R}^{2}(1 + |u^{*}|^{2}) + P_{R}M_{R})e^{M_{R}^{2} + P_{R}M_{R}},\quad  \forall\ t \in [t_{0}, \bar{t}]
\end{equation*}
and we set $Q_{R}:=\bar{C}|\gamma_{x}(t_{0})|^{2}(M_{R}^{2}(1 + |u^{*}|^{2})+ P_{R}M_{R})e^{M_{R}^{2} + P_{R}M_{R}}$. Since $|\gamma_{x}(t_{0})| \in \K$ we have that $|\gamma_{x}(t_{0})| \leq \max\{|y|: y \in \K\}$ and, moreover, $|\gamma_{x}(t)| \leq \max\{|y|: y \in \K\}$ for all times $t$ at which $\gamma_{x}(t) \in \K$, we get the conclusion.

If $t_{0} = 0$, recalling that $\gamma(t_{0})=\gamma(0)=x \in \overline{B}_{R}$ from the above reasoning we get the result. 
\qed

\section{Long-time average and ergodic constant}
\label{sec:ergodicconstant}

 In  this section, we investigate the  existence of the limit
\begin{equation*}
\lim_{T \to \infty} \frac{1}{T}V_{T}(x)	 \quad (x \in \R^{d}),
\end{equation*}
where $V_{T}(x)$ is defined in \eqref{eq:evoValue}. We also address the related problem of the existence of solutions to the ergodic Hamilton-Jacobi equation 
\begin{equation}\label{eq:critical}
c+H(x, D\chi(x)) = 0 \quad (x \in \R^{d})	
\end{equation}
for some $c \in \R$, where we recall that $H: \R^{d} \times \R^{d} \to \R$ is given by  
\begin{equation*}
H(x,p)=\sup_{u \in \R^{m}} \left\{ \langle p, f(x,u) \rangle  - L(x,u) \right\}, \quad \forall\ (x,p) \in \R^{d} \times \R^{d}.
	\end{equation*}
	

In order to prove the main result of this section, \Cref{thm:existencelimit} below, we need to show, first, that the value function $V_{T}$ has locally bounded oscillation uniformly in $T$.

\begin{lemma}\label{lem:boundedoscillation}
	Assume {\bf (F0)}, {\bf (L1)} -- {\bf (L3)} and {\bf (LUGC)}.  For any $R \geq 0$ there exists $K(R) \geq 0$ such that
	\begin{equation*}
		|V_{T}(x) - V_{T}(y)| \leq K(R), \quad \forall\ T \geq T_{R} \quad \forall\ x, y \in \overline{B}_{R},
	\end{equation*}
	where $T_{R}$ is given by {\bf (LUGC)}.
\end{lemma}
\begin{proof}
 Let $R \geq 0$ and let $x$, $y \in \overline{B}_{R}$.
 Let $(\overline\gamma_{y}, \overline{u}_{y}) \in \Gamma_{0, T_{R}}^{y \to x}$ where $T_{R}$ is given by {\bf (LUGC)} and let $(\gamma_{x}, u_{x}) \in \Gamma_{0,T}^{x \to}$ be a solution of \eqref{eq:minimization}. Define the control 
\begin{align*}
\widehat{u}_{y}(t)=
\begin{cases}
	\overline{u}_{y}(t), & 	\quad t \in [0, T_{R}]
	\\
	u_{x}(t-T_{R}), & \quad t \in (T_{R}, T]
\end{cases}	
\end{align*}
and let $\widehat\gamma_{y}$ be the associated trajectory. 
Then, we have that
\begin{align}\label{eq:eq1}
\begin{split}
	& V_{T}(y)-V_{T}(x)
	\\
	 \leq\ & \int_{0}^{T}{L(\widehat\gamma_{y}(t), \widehat{u}_{y}(t))\ dt} - \int_{0}^{T}{L(\gamma_{x}(t), u_{x}(t))\ dt}
\\
\leq\ & \int_{0}^{T_{R}}{L(\overline\gamma_{y}(t), \overline{u}_{y}(t))\ dt}	
+ \int_{T_{R}}^{T}{L(\gamma_{x}(t-T_{R}), u_{x}(t-T_{R}))\ dt} -  \int_{0}^{T}{L(\gamma_{x}(t), u_{x}(t))\ dt}
\\
=\ &  \int_{0}^{T_{R}}{L(\overline\gamma_{y}(t), \overline{u}_{y}(t))\ dt} + \int_{0}^{T-T_{R}}{L(\gamma_{x}(s), u_{x}(s))\ ds} - \int_{0}^{T}{L(\gamma_{x}(s), u_{x}(s))\ ds}  
\\
=\ &  \int_{0}^{T_{R}}{L(\overline\gamma_{y}(t), \overline{u}_{y}(t))\ dt} 
 - \int_{T-T_{R}}^{T}{L(\gamma_{x}(s), u_{x}(s))\ ds}.
 \end{split}
\end{align}
First, by \eqref{eq:L0} we get
\begin{equation*}
\int_{T-T_{R}}^{T}{L(\gamma_{x}(s), u_{x}(s))\ ds} \geq T_{R} L(x^{*},u^{*}).	
\end{equation*}
Then, since $\| \overline{u}_{y} \|_{2, [0,T_{R}]} \leq C_{R}$ by {\bf (LUGC)}, from \Cref{lem:boundedtrajectories} we deduce that 
\begin{equation}\label{eq:pp}
	|\overline\gamma_{y}(t)| \leq (R + c_{f}(1+C_{R})T_{R})e^{c_{f}(1+C_{R})T_{R}}=: \la(R), \quad \forall\ t \in [0, T_{R}].
\end{equation}
Thus, by {\bf (L1)} we obtain
\begin{equation*}
\int_{0}^{T_{R}}{L(\overline\gamma_{y}(s), \overline{u}_{y}(s))\ ds} \leq T_{R}\beta(\la(R))(1+C_{R})=:T_{R} K_{0}(R).	
\end{equation*}
Hence, going back to \eqref{eq:eq1}, we have that
\begin{align*}
& V_{T}(y)-V_{T}(x)
\\
\leq & \int_{0}^{T_{R}}{L(\overline\gamma_{y}(s), \overline{u}_{y}(s))\ ds}- \int_{T-T_{R}}^{T}{L(\gamma_{x}(s), u_{x}(s))\ ds}
\\
\leq\ & T_{R} (K_{0}(R) - L(x^{*},u^{*})):=K(R).
\end{align*}
Switching $x$ and $y$ in the above reasoning completes the proof.\end{proof}

\begin{lemma}\label{lem:periodic}
Assume {\bf (F0)}, {\bf (L1)} -- {\bf (L3)} and {\bf (LUGC)}. For any $R \geq 0$ there exists a constant $N_{R} \geq 0$	such that for any $x \in \overline{B}_{R}$, any $T > T_{R'}$ with $R^{'}=Q_{R} \vee R$, where $Q_{R}$ is given in \Cref{thm:boundednesstraj}, and any optimal pair $(\gamma_{x}, u_{x}) \in \Gamma_{0,T}^{x \to}$ of \eqref{eq:minimization} there exists a pair $(\gamma_{T}, u_{T}) \in \Gamma_{0,T}^{x \to x}$  such that
\begin{equation*}
\int_{0}^{T}{L(\gamma_{T}(t), u_{T}(t))\ dt} \leq \int_{0}^{T}{L(\gamma_{x}(t), u_{x}(t))\ dt} + N_{R}.
\end{equation*}
\end{lemma}
\proof
Fix $R \geq 0$, $x \in \overline{B}_{R}$, and take an optimal pair $(\gamma_{x}, u_{x}) \in \Gamma_{0,T}^{x \to}$. If $\gamma_{x}(T)=x$ then it is enough to take $C_{R}=0$ and $(\gamma_{T}, u_{T})=(\gamma_{x}, u_{x})$. If this is not the case, by \Cref{thm:boundednesstraj} we have that $|\gamma_{x}(t)| \leq Q_{R}$ for any $t \in [0,T]$. So, set
\begin{equation*}
R^{'}=Q_{R} \vee R. 
\end{equation*}
%
For simplicity of notation set $y= \gamma_{x}(T-T_{R'})$. Let $(\bar\gamma_{y}, \bar{u}_{y}) \in \Gamma_{0,T_{R'}}^{y \to x}$ and define the control
\begin{align*}
u_{T}(t)=
\begin{cases}
	u_{x}(t), & \quad t \in [0,T-T_{R'}]
	\\
	\bar{u}_{y}(t+T_{R'}-T), & \quad t \in (T-T_{R'}, T].
\end{cases}	
\end{align*}
Then
\begin{align*}
& \int_{0}^{T}{L(\gamma_{T}(t), u_{T}(t))\ dt}	=  \int_{0}^{T-T_{R'}}{L(\gamma_{x}(t), u_{x}(t))\ dt} + \int_{0}^{T_{R'}}{L(\bar\gamma_{y}(t), \bar{u}_{y}(t))\ dt}
\\
= & \int_{0}^{T}{L(\gamma_{x}(t), u_{x}(t))\ dt} - \int_{T-T_{R'}}^{T}{L(\gamma_{x}(t), u_{x}(t))\ dt} + \int_{0}^{T_{R'}}{L(\bar\gamma_{y}(t), \bar{u}_{y}(t))\ dt}.
\end{align*}
By \eqref{eq:L0} we obtain
\begin{equation*}
	\int_{T-T_{R'}}^{T}{L(\gamma_{x}(t), u_{x}(t))\ dt} \geq T_{R'} L(x^{*},u^{*}).
\end{equation*}
Since by {\bf (LUGC)} we have that $\| \bar{u}_{y} \|^{2}_{2, [0,T_{R'}]} \leq C_{R}$ and $|y| \leq Q_{R}$, by \Cref{lem:boundedtrajectories} we also have that 
\[
|\bar\gamma_{y}(t)| \leq (Q_{R}+ c_{f}(1+C_{R})T_{R'})e^{c_{f}(1+C_{R})T_{R'}}=: \la(R)
\]
for any $t \in [0,T_{R'}]$. So, we obtain
\begin{multline*}
 - \int_{T-T_{R'}}^{T}{L(\gamma_{x}(t), u_{x}(t))\ dt} + \int_{0}^{T_{R'}}{L(\bar\gamma_{y}(t), \bar{u}_{y}(t))\ dt}
 \\
 \leq - T_{R'} L(x^{*},0) + \int_{0}^{T_{R'}}{L(\bar\gamma_{y}(t), \bar{u}_{y}(t))\ dt}  \leq\ T_{R'}\big(\beta(\la(R))(1+C_{R}) - L(x^{*},u^{*})\big)
\end{multline*}
where the last inequality holds by {\bf (L1)}. 
The conclusion follows taking
\begin{equation*}
N_{R} = T_{R'}\big(\ell_{1}\beta(\la(R))(1+C_{R}) - L(x^{*},u^{*})\big).	 \eqno{\square}
\end{equation*}

\begin{theorem}[{\bf Existence of the critical constant}]\label{thm:existencelimit}
Assume {\bf (F0)}, {\bf (L1)} -- {\bf (L3)} and {\bf (LUGC)}. There exists a constant $\mane \in \R$, called the critical constant (or Ma\~n\'e's critical value), such that
\begin{equation}\label{eq:existencemane}
\lim_{T \to \infty} \sup_{x \in \overline{B}_{R}} \left|\frac{1}{T}V_{T}(x) - \mane\right| = 0, \quad \forall\ R > 0.
\end{equation}
\end{theorem}
\proof

Let $R \geq 0$. By \Cref{lem:boundedoscillation}, for all $x \in \overline{B}_{R}$ we deduce that
\begin{equation}\label{eq:boundR}
|V_{T}(x) - V_{T}(0)| \leq K(R). 	
\end{equation}
Hence, to obtain the conclusion it suffices to prove the existence of the limit 
\begin{equation}\label{eq:limitmane}
\lim_{T \to \infty} \frac{1}{T} V_{T}(0)=: \mane. 	
\end{equation}
For this purpose let $\{ T_{n}\}_{n \in \N}$ and $\{(\gamma_{n}, u_{n})\}_{n \in \N} \subset \Gamma_{0,T_{n}}^{0 \to}$ be such that
\begin{align}\label{eq:limit}
\begin{split}
\liminf_{T \to \infty} \frac{1}{T}V_{T}(0) = & \lim_{n \to \infty} \frac{1}{T_{n}} \inf_{(\gamma, u) \in \Gamma_{0,T_{n}}^{0 \to}} \int_{0}^{T_{n}}{L(\gamma(t), u(t))\ dt}
\\
= & \lim_{n \to \infty} \frac{1}{T_{n}}\int_{0}^{T_{n}}{L(\gamma_{n}(t), u_{n}(t))\ dt}	.
\end{split}
\end{align}
By \Cref{lem:periodic} there exists a sequence $(\gamma_{n}^{0}, u_{n}^{0}) \in \Gamma_{0,T_{n}}^{0 \to 0}$ and a constant $N_{0} \geq 0$ such that
\begin{equation}\label{eq:periodicine}
\int_{0}^{T_{n}}{L(\gamma^{0}_{n}(t), u^{0}_{n}(t))\ dt} \leq \int_{0}^{T_{n}}{L(\gamma_{n}(t), u_{n}(t))\ dt} + N_{0}.
\end{equation}
Next, for any $n \in \N$ let $\widehat{u}^{0}_{n}$ be the periodic extension of $u^{0}_{n}$, i.e., $\widehat{u}^{0}_{n}$ is $T_{n}$-periodic and $\widehat{u}^{0}_{n}(t)=u^{0}_{n}(t)$ for any $t \in [0,T_{n}]$. Then, we have that
\begin{align}\label{eq:limsup}
 & \limsup_{T \to \infty} \frac{1}{T}V_{T}(0) \leq \limsup_{T \to \infty} \frac{1}{T}\int_{0}^{T}{L(\widehat\gamma^{0}_{n}(t), \widehat{u}^{0}_{n}(t))\ dt}, \quad \forall\ n \in \N
\end{align}
by using $\widehat{u}^{0}_{n}$ to estimate from above \eqref{eq:evoValue}. Then, by periodicity and \eqref{eq:periodicine} we obtain 
\begin{multline*}
 \limsup_{T \to \infty} \frac{1}{T}\int_{0}^{T}{L(\widehat\gamma^{0}_{n}(t), \widehat{u}^{0}_{n}(t))\ dt}
 \\ =\ \frac{1}{T_{n}}\int_{0}^{T_{n}}{L(\gamma^{0}_{n}(t), u^{0}_{n}(t))\ dt} 
 \leq \frac{1}{T_{n}}\int_{0}^{T_{n}}{L(\gamma_{n}(t), u_{n}(t))\ dt} + \frac{N_{0}}{T_{n}}.	
\end{multline*}
Therefore, recalling \eqref{eq:limsup} and \eqref{eq:limit} we conclude that
\begin{align*}
\limsup_{T \to \infty} \frac{1}{T}V_{T}(0) \leq\ & \lim_{n \to \infty} \left(\frac{1}{T_{n}}\int_{0}^{T_{n}}{L(\gamma_{n}(t), u_{n}(t))\ dt} + \frac{N_{0}}{T_{n}}\right)
\\
=\ & \liminf_{T \to \infty} \frac{1}{T}V_{T}(0).	
\end{align*}
This yields \eqref{eq:limitmane}, thus completing the proof. \qed

\begin{corollary}\label{cor:minmane}
	Assume {\bf (F0)}, {\bf (L1)} -- {\bf (L3)} and {\bf (LUGC)}. Then, we have that 
	\begin{equation*}
	\mane=\min_{(x, u) \in \R^{d} \times \R^{m}} L(x,u).	
	\end{equation*}
\end{corollary}
\proof
First, we recall that 
\begin{equation*}
	\mane=\lim_{T \to \infty} \frac{1}{T}V_{T}(0)=\lim_{T \to \infty} \frac{1}{T} \inf_{(\gamma,u) \in \Gamma_{0,T}^{0 \to}} \int_{0}^{T}{L(\gamma(s), u(s))\ ds}.
\end{equation*}
So, taking $(\gamma_{x}, u_{x}) \in \Gamma_{0,T}^{0\to}$ optimal for $V_{T}(0)$ we obtain 
\begin{align*}
\mane=\ & \lim_{T \to \infty} \frac{1}{T} \int_{0}^{T}{L(\gamma_{x}(s), u_{x}(s))\ ds} \geq\ \frac{1}{T}\lim_{T \to \infty} \int_{0}^{T}{L(x^{*},u^{*})\ ds}
\\
=\ & L(x^{*},u^{*})=\min_{(x,u) \in \R^{d} \times \R^{m}} L(x,u). 	
\end{align*}
since, by assumption {\bf (L2)} and {\bf (L3)}, we have that $L(x,u) \geq L(x^{*},u^{*})$ for any $(x,u) \in \R^{d} \times \R^{m}$. 

On the other hand, we observe that, owing to \Cref{thm:existencelimit}, the value of $\mane$ could be computed replacing $0$ in \eqref{eq:limitmane} with any other point of $\R^d$.
So,
\begin{equation*}
\mane=\lim_{T\to\infty}\frac{1}{T}V_{T}(x^{*}).
\end{equation*}
Now, taking the control $u(t)=u^{*}$ which makes $x^{*}$ stationary yields 
\begin{equation*}
\mane = \lim_{T \to \infty} \frac{1}{T} \inf_{(\gamma,u) \in \Gamma_{0,T}^{x^{*} \to}} \int_{0}^{T}{L(\gamma(s), u(s))\ ds} \leq \lim_{T\to\infty}\frac{1}{T}\int_{0}^{T}{L(x^{*},u^{*})\ ds}=L(x^{*},u^{*})
\end{equation*}
which in turn implies the conclusion. \qed

\begin{remarks}\label{rem:min0}\em
Note that in view of \Cref{thm:existencelimit} we have that 
\begin{equation*}
\lim_{T \to \infty} \sup_{x \in \overline{B}_{R}} \left| \frac{1}{T}\inf_{(\gamma, u) \in \Gamma_{0,T}^{x \to}} \int_{0}^{T}{\big(L(\gamma(s), u(s)) - \mane \big)\ ds} \right| = 0, \quad \forall\ R \geq 0.
\end{equation*}
Moreover, from \Cref{cor:minmane} we deduce that  
\begin{equation*}
\min_{(x, u) \in \R^{d} \times \R^{m}} L(x,u) - \mane = 0.
\end{equation*}
Therefore, by replacing $L$ with $\widehat{L}(x,u):=L(x,u) - \mane$ one can reduce the analysis to the case of $\alpha(\widehat{L})=\min_{(x, u) \in \R^{d} \times \R^{m}} \widehat{L}(x,u)=0$.\end{remarks}

\section{Sub-Riemannian control systems: application to Abel means}
\label{sec:Abel}
	In this section, we concentrate our analysis to the sub-Riemannian control systems for which we know that {\bf (LUGC)} holds, \Cref{ex:exam}. Now, having at our disposal the existence of the critical constant $\mane$ for such systems, we will construct a continuous viscosity solution to the ergodic equation 
	\begin{equation*}
	\mane + H(x, D\chi(x))=0, \quad (x \in \R^{d}).
	\end{equation*}

	 For $m \in \N$ and $i=1, \dots, m$, let 
	\begin{equation*}
		f_{i}: \R^{d} \to \R^{d} 
	\end{equation*}
	and
	\begin{equation*}
		u_{i}:[0,\infty) \to \R
	\end{equation*}
be smooth vector fields and measurable controls, respectively, and consider the following controlled dynamics of sub-Riemannian type
	\begin{align}\label{eq:dynamics1}
         \dot\gamma(t)= \displaystyle{\sum_{i=1}^{m}{f_{i}(\gamma(t))u_{i}(t)}}= F(\gamma(t))u(t), \quad t \in [0,+\infty)
	\end{align}
	where $F(x)=[f_{1}(x)| \dots| f_{m}(x)]$ is an $d \times m$ real matrix and $u(t)=(u_{1}(t), \dots, u_{m}(t))^{\star}$\footnote{$(u_{1}, \dots, u_{m})^{\star}$ denotes the transpose of $(u_{1}, \dots, u_{m})$}. 
We assume the vector fields $f_{i}$ to satisfy the following.
\begin{itemize}
\item[({\bf F1})] There exists a constant $c_{f} \geq 1$ such that for any $i=1,\dots, m$
\begin{equation}\label{eq:fassum}
|f_{i}(x)| \leq c_{f}(1+|x|), \quad \forall x \in \R^{d}.	
\end{equation}
\end{itemize}
Note that, by {\bf (F1)} system \eqref{eq:dynamics1} satisfies {\bf (F0)}.

Moreover, we assume the following.
\begin{itemize}
\item[{\bf (L3')}] There exists a compact set $\mathcal{K} \subset \R^{d}$ such that 
\begin{equation*}
\min_{x \in \mathcal{K}} L(x,0) = 0, \quad \text{and} \quad  \inf_{x \in \R^{d} \backslash \mathcal{K}} L(x,0) > 0.
\end{equation*} 
\end{itemize}
We recall that in view of \Cref{rem:min0} assumption {\bf (L3')} is not restrictive and, moreover, by \Cref{cor:minmane} we have that $\mane = 0$. Furthermore, {\bf (L3')} subsumes {\bf (L2)} and {\bf (L3)} given so far. 

Let $x^{*} \in \K$ be such that 
\begin{equation*}
L(x^{*},0)=\min_{x \in \K} L(x,0)
\end{equation*}
and set
\begin{equation*}
	\delta^{*}(x)=d_{\SR}(x,x^{*}), \quad \forall\ x \in \R^{d}.
\end{equation*}
Observe that, by \Cref{cor:ballbox}, there exists a nondecreasing function $D: \R_{+} \to \R_{+}$ with
\begin{equation}\label{eq:ballboxincrease}
\delta^{*}(x) \leq D(|x|), \quad \forall\ x \in \R^{d}.	
\end{equation}

Next, we provide estimates on trajectories satisfying \eqref{eq:dynamics1}. 
\begin{lemma}\label{lem:L2bounds}
	Let $x \in \R^{d}$, $t \geq 0$ and $(\gamma, u) \in \Gamma_{0,t}^{x \to}$. Then there exists a constant $\kappa(\|u\|_{2}, t) \geq 0$ such that
	\begin{equation}\label{eq:00}
		|\gamma(s)|\leq \kappa(\|u\|_{2},t)(1+|x|), \quad \forall\ s \in [0,t]
	\end{equation}
and 
\begin{equation}\label{eq:000}
|\gamma(t_{2}) - \gamma(t_{1})| \leq c_{f}\kappa(\|u\|_{2},t)(1+|x|)\|u\|_{2} |t_{2}-t_{1}|^{\frac{1}{2}}, \quad 0 \leq t_{1} \leq t_{2} \leq t. 	
\end{equation}
\end{lemma}
\proof
We begin by proving \eqref{eq:00}. For any $s \in [0,t]$ we have that
\begin{align*}
|\gamma(s)| \leq\ & |x| + \int_{0}^{t}{|F(\gamma(s))||u(s)|\ ds}
\\
\leq\ & |x| + \int_{0}^{t}{c_{f}(1+ |\gamma(s)|)|u(s)|\ ds}
\\
\leq\ & |x| + c_{f}\left(\int_{0}^{t}{\big(1+|\gamma(s)|\big)^{2}\ ds}\right)^{\frac{1}{2}}\| u\|_{2}.
\end{align*}
Thus, we get
\begin{align*}
|\gamma(s)|^{2} \leq C\left(|x|^{2} + c_{f}^{2}t\|u\|_{2}^{2} + c_{f}^{2}\|u\|_{2}^{2}\int_{0}^{t}{|\gamma(s)|^{2}\ ds}  \right)	
\end{align*}
which implies the \eqref{eq:00} by Gronwall's inequality.

We now proceed to show \eqref{eq:000}. Fix $t_{1}$, $t_{2}$ such that $0 \leq t_{1} \leq t_{2} \leq t$. Then, we have that 
\begin{align*}
	|\gamma(t_{2})-\gamma(t_{1})| \leq\ & \int_{t_{1}}^{t_{2}}{|F(\gamma(s))||u(s)|\ ds}
	\\
	\leq\ & c_{f}\int_{t_{1}}^{t_{2}}{(1+|\gamma(s)|)|u(s)|\ ds}
	\\
	\leq\ & c_{f}\kappa(\|u\|_{2},t)(1+|x|)\int_{t_{1}}^{t_{2}}{|u(s)|\ ds}
\end{align*}
where the last inequality holds by \eqref{eq:00}. Hence, by H\"older's inequality we obtain 
\begin{align*}
	|\gamma(t_{2})-\gamma(t_{1})| \leq c_{f}\kappa(\|u\|_{2},t)(1+|x|)\|u\|_{2}|t_{2}-t_{1}|^{\frac{1}{2}}. 
\end{align*}
This completes the proof of the lemma. \qed

Now, we move to the analysis of the ergodic equation 
\begin{equation*}
\mane+H(x, D\chi)=0, \quad x \in \R^{d}.	
\end{equation*}
We will show the existence of viscosity solutions to such an equation by studying the limit behavior of solutions to the discounted problem
\begin{equation}\label{eq:discounted}
\lambda v_{\lambda}(x) + H(x, Dv_{\lambda}(x))=0, \quad x \in \R^{d}	
\end{equation}
as $\lambda\downarrow 0$. 
To do so, define the function
	\begin{equation}\label{eq:discountedvalue}
	v_{\lambda}(x)=\inf_{(\gamma, u) \in \Gamma_{0,\infty}^{x \to}(e^{-\lambda t}dt)}\left\{\int_{0}^{+\infty}{e^{-\lambda t}L(\gamma(t), u(t))\ dt}\right\},	
	\end{equation}
where
\begin{align*}
	\Gamma_{0,\infty}^{x \to}(e^{-\lambda t}dt)  := \Big\{ & (\gamma, u) \in L^{\infty}_{\text{loc}}(0,\infty; \R^{d}) \times L^{2}_{\text{loc}}(0,\infty; \R^{m}):
	\\
	& (\gamma, u) \in \Gamma_{0, T}^{x \to} \,\,\,\, \forall\ T > 0,\,\, \text{and}\,\, \int_{0}^{\infty}{e^{-\lambda t}|u(t)|^{2}\ dt} < \infty\Big\}.
\end{align*}

	Note that, by {\bf (L0)} and {\bf (L3')} we have that $v_{\lambda}(x) \geq 0$ for any $x \in \R^{d}$. Moreover, $v_{\lambda}$ is the continuous viscosity solution of \eqref{eq:discounted}. 
 \begin{proposition}\label{prop:uniform}
Assume {\bf (F1)}, and  {\bf (L1)} -- {\bf (L3')}. Then, for any $R \geq 0$ we have that:
\begin{itemize}
\item[($i$)] $\{ \lambda v_{\lambda}\}_{\lambda >0}$ is equibounded on $\overline{B}_{R}$;
 \item[($ii$)] there exists a constant $C_{R} \geq 0$ such that
\begin{equation}\label{eq:appholder}
|v_{\lambda}(x) - v_{\lambda}(y)| \leq C_{R}d_{\SR}(x,y), \quad \forall\ x, y \in \overline{B}_{R}.
\end{equation}
\end{itemize}
\end{proposition}

\begin{remarks}\em
Recalling that  $r_{R}$ is the uniform degree of nonholonomy of the distribution $\{f_{i}\}_{i=1,\dots,m}$ associated with the compact set $\overline{B}_{R}$, \Cref{cor:ballbox} and \eqref{eq:appholder} yield
\begin{equation*}
|v_{\lambda}(x)-v_{\lambda}(y)| \leq C_{R}\widetilde{c}_{2}|x-y|^{\frac{1}{r_{R}}} \quad \forall\ x,y \in \overline{B}_{R}.
\end{equation*}
\end{remarks}

\noindent{\it Proof of \Cref{prop:uniform}:}
	Let $R \geq 0$ and let $x \in \overline{B}_{R}$. Taking $(\bar\gamma, \bar{u}) \in \Gamma_{0,\infty}^{x \to}(e^{-\lambda t}dt)$ such that $(\bar\gamma(t), \bar{u}(t))  \equiv (x, 0)$, by {\bf (L1)} we get
	\begin{align*}
	 \lambda v_{\lambda}(x) \leq\ & \lambda \int_{0}^{+\infty}{e^{-\lambda t}L(x, 0)\ dt}
	\\
	 \leq\ &   \beta(R)\int_{0}^{+\infty}{\lambda e^{-\lambda t}\ dt}=\beta(R).	\end{align*}
On the other hand, by {\bf (L3')} we have that
\begin{align*}
& \lambda v_{\lambda}(x) \geq 0.
\end{align*}
Thus, for any $\lambda > 0$ we conclude that  
\begin{equation*}
\lambda |v_{\lambda}(x)| \leq \beta(R), \quad \forall\ x \in \overline{B}_{R}. 	
\end{equation*}

	In order to prove ($ii$), for any fixed $x$, $y \in \overline{B}_{R}$ set $\delta=d_{\SR}(x,y)$. Let $(\bar\gamma_{y}, \bar{u}_{y}) \in \Gamma_{0, \delta}^{y \to x}$ be a solution of \eqref{eq:subriem}. Let $(\gamma_{x}, u_{x}) \in \Gamma_{0,+\infty}^{x \to}(e^{-\lambda t}dt)$ be such that 
	\begin{equation*}
	\int_{0}^{\infty}{e^{-\lambda t}L(\gamma_{x}(t), u_{x}(t))\ dt} \leq v_{\lambda}(x) + \lambda.	
	\end{equation*}
Define a new control 
\begin{align*}
\widehat{u}_{y}(t)=
\begin{cases}
	\bar{u}_{y}(t), & 	\quad t \in [0, \delta]
	\\
	u_{x}(t-\delta), & \quad t \in (\delta, +\infty),
\end{cases}	
\end{align*}
and so $(\widehat\gamma_{y}, \widehat{u}_{y}) \in \Gamma_{0,\infty}^{y \to}(e^{-\lambda t}dt)$. Then, we have that
\begin{align*}
& v_{\lambda}(y)-v_{\lambda}(x) 
\\
\leq & \int_{0}^{\delta}{e^{-\lambda t}L(\bar\gamma_{y}(t), \bar{u}_{y}(t))\ dt}	
+ \int_{\delta}^{+\infty}{e^{-\lambda t}L(\gamma_{x}(t-\delta), u_{x}(t-\delta))\ dt} 
\\
-\ &  \int_{0}^{+\infty}{e^{-\lambda t}L(\gamma_{x}(t), u_{x}(t))\ dt} + \lambda
\\
=\ &  \int_{0}^{\delta}{e^{-\lambda t}L(\bar\gamma_{y}(t), \bar{u}_{y}(t))\ dt} 
+ (e^{-\lambda \delta}-1)\int_{0}^{+\infty}{e^{-\lambda s}L(\gamma_{x}(s), u_{x}(s))\ ds} + \lambda
\\
 =\ & \int_{0}^{\delta}{e^{-\lambda t}L(\bar\gamma_{y}(t), \bar{u}_{y}(t))\ dt} + \big(-\delta\lambda  + \text{o}(\delta\lambda)\big)(v_{\lambda}(x) + \lambda) + \lambda
\end{align*}
where 
\begin{equation*}\lim_{q \to 0} \frac{\text{o}(q)}{q}= 0.	
\end{equation*}
By point ($i$) we have that $\delta\lambda v_{\lambda}(x) \leq \delta \beta(R)$ and for $\lambda \leq 1$ we obtain $o(\delta\lambda) \leq \delta$. Moreover, by \Cref{lem:boundedtrajectories} we know that $$|\bar\gamma_{y}(t)|\leq (|y| + c_{f}\delta)e^{c_{f}\delta}=:\la(R), \quad \forall\ t \in [0,\delta]$$ 
since $\|\bar{u}_{y}\|_{\infty, [0,\delta]} \leq 1$. 
Thus, by {\bf (L2)} we deduce that
\begin{align*} 
\int_{0}^{\delta}{e^{-\lambda t}L(\bar\gamma_{y}(t), \bar{u}_{y}(t))\ dt} \leq \int_{0}^{\delta}{\beta(|\bar\gamma_{y}(t)|)(1+|\bar{u}_{y}(t)|^{2})\ dt}\leq  2 \delta \beta(\la(R)).
\end{align*} 
Therefore, setting $C_{R}=2 \beta(\la(R))$ we obtain \eqref{eq:appholder} recalling that $\delta=d_{\SR}(x,y).$ \qed

 Note that, the above proof fails for general control systems, i.e., of the form \eqref{eq:dynamics}, under the assumption {\bf (LUGC)} since, a priori, $T_{R}$ might not be of the order of $|x-y|$.


\begin{theorem}[{\bf Existence of correctors}]\label{thm:convergence}
		Assume {\bf (F1)}, {\bf (F2)} and  {\bf (L0)} -- {\bf (L3')}. Then there exists a continuous function $\chi:\R^d\to\R$ and a sequence $\lambda_{n} \downarrow 0$ such that, for any $R \geq 0$, 
\begin{equation*}
\lim_{n \to \infty} v_{\lambda_{n}}(x)= \chi(x), \quad \text{uniformly on}\ \overline{B}_{R}.	
\end{equation*}
 Moreover, we have that:
 \begin{itemize}
   \item[($i$)] $\chi(x) \geq 0$, $\chi(x^{*})=0$, and $\chi$ is locally Lipschitz continuous w.r.t. $d_{\SR}$, that is, for any $R \geq 0$ there exists a constant $\ell_{R} \geq 0$ such that 
    \begin{equation}\label{eq:SRlip}
    	|\chi(x)-\chi(y)| \leq \ell_{R}d_{\SR}(x,y), \quad \forall\ x,y \in \overline{B}_{R}.
    \end{equation}
	\item[($ii$)] $\chi$ is a viscosity solution of the ergodic Hamilton-Jacobi equation
\begin{equation}\label{eq:HJ}
H(x, D \chi(x))=0 \quad (x \in \R^{d}).
\end{equation}
 \end{itemize}
 \end{theorem}
\proof
First, we observe that, by an adaptation of \cite[Theorem 5]{bib:MA} (see \Cref{thm:newabelian} in 	\Cref{sec:Appendix1}), we have that
\begin{equation}\label{eq:abelian}
0= \lim_{T \to +\infty} \frac{1}{T}V^{T}(x) = \lim_{\lambda \to 0} \lambda v_{\lambda}(x)
\end{equation}
locally uniformly in space.
We recall that $v_{\lambda}(x)$  is a continuous viscosity solution of
\begin{equation*}
	\lambda v_{\lambda}(x) + H(x, Dv_{\lambda}(x))=0 \quad (x \in \R^{d})
\end{equation*}
Since $v_{\lambda}(x^{*})=0$, by \Cref{prop:uniform} we deduce that $\{v_{\lambda}\}_{\lambda > 0}$ is equibounded and equicontinuous. So, applying the Ascoli-Arzel\'a Theorem and a diagonal argument  we deduce that there exists a sequence  $\lambda_{n} \downarrow 0$ such that $\{ v_{\lambda_{n}}(x)\}_{n \in \N}$ is locally uniformly convergent, i.e., for any $R \geq 0$
\begin{equation*}
\lim_{n \to \infty} v_{\lambda_{n}}(x) =: \chi(x)\quad\mbox{uniformly on}\quad \overline{B}_{R}. 	
\end{equation*}
Hence,  from {\bf (L3')} we immediately deduce that $\chi(x) \geq 0$ and, again, since $v_{\lambda}(x^{*})=0$ we get $\chi(x^{*})=0$. Furthermore, from  \eqref{eq:appholder} we get \eqref{eq:SRlip}. 
Finally,  the stability of viscosity solutions ensures that $\chi$ is a  solution of \eqref{eq:HJ}, which proves ($ii$). \qed

\begin{definition}[{\bf Critical equation and critical solutions}]
The equation
	\begin{equation}\label{eq:criticaleq}
H(x, D\chi(x))= 0 \quad (x \in \R^{d})	
\end{equation}
is called the {\em critical} (or, {\em ergodic}) Hamilton-Jacobi equation. A continuous function $\chi$ is called a {\em critical subsolution} (resp. {\em supersolution}) if it is a viscosity subsolution (resp. supersolution) of \eqref{eq:criticaleq} and a {\em critical solution} if it is both a subsolution and a supersolution. 
\end{definition}

\section{Representation formula}
\label{sec:LaxOleinik}

In this last section, we construct a critical solution that can be represented as the value function of a sub-Riemannian optimal control problem. Such a  solution, which is useful to develop the Aubry-Mather theory in the sub-Riemannian case, will be obtained as the asymptotic limit as $t\to\infty$ of the Lax-Oleinik semigroup, applied to $\chi$  given by \Cref{thm:convergence}. 

We begin by giving the definition of dominated functions. 
\begin{definition}[{\bf Dominated functions}]\label{def:domcal}
			Let $a$, $b \in \R$ such that $a < b$ and let $x$, $y \in \R^{d}$.  Let $\phi$ be a continuous function on $\R^{d}$.		
			We say that $\phi$ is dominated by $L-c$, and we denote this by $\phi \prec L -c$, if for any trajectory-control pair $ (\gamma,u) \in \Gamma_{a,b}^{x \to y}$ we have that
			\begin{equation*}
			\phi(y) - \phi(x) \leq \int_{a}^{b}{L(\gamma(s), u(s))\ ds} - c\ (b-a).
			\end{equation*}
\end{definition}

Let us introduce, now, the  following class of functions 
 \begin{equation*}
\mathcal{S}=\bigg\{\varphi \in C(\R^{d}): \varphi(x) \geq 0\,\,\, \forall\ x \in \R^{d},\,\, \varphi \prec L \bigg\}
\end{equation*}
endowed with the topology induced by the uniform convergence on compact sets. Then, for any $x \in \R^{d}$, any $t \geq 0$, and any $\varphi \in \mathcal{S}$  define the functional 
\begin{equation*}
	\mathcal{F}_{\varphi}: \Gamma_{0,t}^{\to x} \to \R
\end{equation*}
as
\begin{equation*}
\mathcal{F}_{\varphi}(\gamma, u)=	 \varphi(\gamma(0)) + \int_{0}^{t}{L(\gamma(s), u(s))\ ds}
\end{equation*}
and set
\begin{align}\label{eq:LAsemi}
T_{t}\varphi(x) = \inf_{(\gamma,u) \in \Gamma_{0,t}^{\to x}} \left\{\varphi(\gamma(0)) + \int_{0}^{t}{L(\gamma(s), u(s))\ ds} \right\}.
\end{align}
Before proceeding  to derive several properties of $T_{t}\varphi$, including the fact that $T_{t}\varphi(x) \geq 0$, we first show that the class $\mathcal{S}$ is non-empty. 

\begin{lemma}
	Assume {\bf (F1)}, {\bf (F2)} and  {\bf (L0)} -- {\bf (L3')}. Then, the function $\chi$ constructed in \Cref{thm:convergence} belongs to $\mathcal{S}$.  
\end{lemma}
\proof
Let $\chi$ be the critical solution given in \Cref{thm:convergence}, i.e., 
\begin{equation*}
\chi(x)=\lim_{n \to \infty} v_{\lambda_{n}}(x)	
\end{equation*}
where the limit is uniform on all compact subsets of $\R^{d}$. Recall that for any $\lambda > 0$
\begin{equation*}
v_{\lambda}(x)=\inf_{(\gamma,u) \in \Gamma_{0,\infty}^{x \to}(e^{\lambda t}\ dt)}\int_{0}^{\infty}{e^{-\lambda t}L(\gamma(t), u(t))\ dt}, \quad (x \in \R^{d}). 
\end{equation*}
From \Cref{thm:convergence} we know that $\chi$ is continuous and $\chi(x) \geq 0$ for any $x \in \R^{d}$. 
Hence, we only need to prove that $\chi \prec L.$ To do so, let $R \geq 0$ and let $x$, $y \in \overline{B}_{R}$. Fix $a$, $b \in \R$ and let $(\gamma,u) \in \Gamma_{a,b}^{x \to y}$. Let $(\gamma_{y}, u_{y}) \in \Gamma_{0,\infty}^{y \to}(e^{-\lambda t}\ dt)$ be $\lambda$-optimal for $v_{\lambda}(y)$, that is, 
\begin{equation*}
	\int_{0}^{\infty}{e^{-\lambda t}L(\gamma_{y}(t), u_{y}(t))\ dt} \leq v_{\lambda}(y) + \lambda.	
	\end{equation*}
and define the control 
\begin{equation*}
	\widetilde{u}(t)=
\begin{cases}
u(t+a), & t \in [0,b-a]
\\
u_{y}(t+a-b), & t \in (b-a, \infty).
\end{cases}	
\end{equation*}
Then, $(\widetilde\gamma, \widetilde{u}) \in \Gamma_{0,\infty}^{x \to}$ and $\widetilde\gamma(t)=\gamma_{y}(t+a-b)$ for all $t \geq b-a$. Therefore,
\begin{align*}
v_{\lambda}(x)-v_{\lambda}(y) \leq\ & \int_{0}^{\infty}{e^{-\lambda t}L(\widetilde\gamma(t), \widetilde{u}(t))\ dt} - \int_{0}^{\infty}{e^{-\lambda t}L(\gamma_{y}(t), u_{y}(t))\ dt} + \lambda
\\
=\ & \int_{0}^{b-a}{e^{-\lambda t}L(\gamma(t+a), u(t+a))\ dt} 
\\
+\ & \int_{b-a}^{\infty}{e^{-\lambda t}L(\widetilde\gamma(t), \widetilde{u}(t))\ dt}	- \int_{0}^{\infty}{e^{-\lambda t}L(\gamma_{y}(t), u_{y}(t))\ dt} + \lambda
\\
\leq\ & \int_{a}^{b}{L(\gamma(t), u(t))\ dt} + \left(e^{-\lambda(b-a)}-1 \right)\int_{0}^{\infty}{e^{-\lambda t}L(\gamma_{y}(t), u_{y}(t))\ dt} + \lambda
\\
=\ & \int_{a}^{b}{L(\gamma(t), u(t))\ dt} -(b-a)\lambda v_{\lambda}(y)(1+o(1)) + \lambda.
\end{align*}
So, taking $\lambda=\lambda_{n}$ in the previous estimate and passing to the limit we conclude that 
\begin{equation*}
\chi(x)-\chi(y) \leq \int_{a}^{b}{L(\gamma(t), u(t))\ dt}, 
\end{equation*}
which completes the proof. \qed

\begin{theorem}[{\bf Lax-Oleinik semigroup}]\label{thm:SRbounds}
	Assume {\bf (F1)}, {\bf (F2)} and  {\bf (L0)} -- {\bf (L3')}. Then, the following holds.
	\begin{enumerate}
	    \item For any $\varphi \in \mathcal{S}$ there exists a function $N_{\varphi}: \R^{d} \to \R$, bounded on all compact sets, such that for any $(t,x) \in [0, \infty) \times \R^{d}$ there exists $(\gamma_{x}, u_{x}) \in \Gamma_{0,t}^{\to x}$ satsfying
	    \begin{equation}\label{eq:sublevel}
	    \mathcal{F}_{\varphi}(\gamma_{x},u_{x}) \leq N_{\varphi}(x).
	     \end{equation}
		\item For any $\varphi \in \mathcal{S}$ there exists a nondecreasing function $C_{\varphi}: [0,\infty) \to [0,\infty)$ such that for any $R \geq 0$, any $(t,x) \in [0, \infty) \times \overline{B}_{R}$, and any $(\gamma,u) \in \Gamma_{0,t}^{\to x}$ satisfying \eqref{eq:sublevel} we have that
	\begin{equation}\label{eq:SRdist}
	d_{\SR}(x, \gamma(0)) \leq C_{\varphi}(R).
	\end{equation}
	Moreover, one can take
	\begin{equation*}
	C_{\varphi}(R):= \beta(R)D(R) + \max_{x \in \overline{B}_{R}} \varphi(x).
	\end{equation*}
	\item For any $(t,x) \in [0,\infty) \times \R^{d}$ and any $\varphi \in \mathcal{S}$ the infimum in \eqref{eq:LAsemi} is attained\footnote{Notice that the existence of minimizing pairs $(\gamma, u) \in \Gamma_{0,t}^{\to x}$ follows by classical results in optimal control theory (see, for instance, \cite[Theorem 7.4.4]{bib:SC}).} and we have that $T_{t}\varphi(x) \geq 0$. 
	\item For any $\varphi \in \mathcal{S}$ and any $c\in\R$ we have that $T_{t}(\varphi + c)=T_{t}\varphi+c$  for all $t\geq 0$. 
	\item $T_{t}$ is a semigroup on $\mathcal{S}$, i.e., $T_{t}: \mathcal{S} \to \mathcal{S}$ and for any  $s$, $t \geq 0$ and any $\varphi \in \mathcal{S}$ 
	\begin{align*}
	T_{0}\varphi= \varphi, \quad 
	T_{s}(T_{t}\varphi)=\ T_{s+t}\varphi.
	\end{align*} 
	\item $T_{t}$ is continuous on $\mathcal{S}$ w.r.t. the topology induced by the uniform convergence on all compact subsets.
	\end{enumerate}
\end{theorem}
\begin{remarks}\em
	 We recall that, a set $\mathcal{K}$ is compact in $(\R^{d}, d_{\SR})$ if and only if $\mathcal{K}$ is compact in $\R^{d}$ w.r.t. the Euclidean distance (see, e.g., \cite[Theorem 3.31]{bib:ABB}).
\end{remarks}
\proof
We begin by proving $(1)$.  To do so, we consider two cases: first, we take $(t, x) \in [D(|x|),\infty) \times \R^{d}$ and, then, $(t, x) \in [0,D(|x|)) \times \R^{d}$. Recall that $D(\cdot)$ is defined in \eqref{eq:ballboxincrease} and satisfies $\delta^{*}(x) \leq D(|x|)$. 

Define the function $N_{\varphi}: \R^d \to \R$ as 
\begin{equation*}
N_{\varphi}(x)=	
\begin{cases}
	\varphi(x^{*}) + D(|x|)\beta(\la(|x|)), & (t,x) \in [D(|x|), \infty) \times \R^{d}
	\\
	\varphi(x) + D(|x|) \beta(|x|), & (t,x) \in [0,D(|x|)) \times \R^{d}
\end{cases}
\end{equation*}
where $\la(|x|):=(|x^{*}| + c_{f}\delta^{*}(x))e^{c_{f}\delta^{*}(x)}$.
Note that, since $\varphi \in \mathcal{S}$ we deduce that $N_{\varphi}$ is bounded on any compact subset of $\R^{d}$. 

We now proceed with the first part of the proof, i.e., we show that for any $(t, x) \in [D(|x|),\infty) \times \R^{d}$ there exists $(\gamma_{x}, u_{x}) \in \Gamma_{0,t}^{\to x}$ such that 
\begin{equation*}
\mathcal{F}_{\varphi}(\gamma_{x},u_{x}) \leq N_{\varphi}(x).	
\end{equation*}
Let $(\gamma_{0},u_{0}) \in \Gamma_{0,\delta^{*}(x)}^{x^{*}\to x}$ be optimal for \eqref{eq:subriem} and define the control 
\begin{equation*}
u_{x}(s)=
\begin{cases}
0, & s \in [0,t-\delta^{*}(x))
\\
u_{0}(s-t+\delta^{*}(x)), & s \in [t-\delta^{*}(x),t]
\end{cases}	
\end{equation*}
so that $(\gamma_{x}, u_{x}) \in \Gamma_{0,t}^{x^{*} \to x}$. Then
\begin{align*}
 \mathcal{F}_{\varphi}(\gamma_{x},u_{x}) =\ & \varphi(x^{*})  + \int_{t-\delta^{*}(x)}^{t}{L(\gamma_{0}(s-t+\delta^{*}(x)), u_{0}(s-t+\delta^{*}(x)))\ ds} 
	 \\ 
	 =\  & \varphi(x^{*}) + \int_{0}^{\delta^{*}(x)}{L(\gamma_{0}(s), u_{0}(s))\ ds} \leq\  \varphi(x^{*}) + \int_{0}^{\delta^{*}(x)}{L(\gamma_{0}(s), u_{0}(s))\ ds}
\end{align*}
Let us estimate the rightmost term above. Recalling that $|u_{0}(s)| \leq 1$ for any $s \in [0,\delta^{*}(x)]$ we have that 
\begin{equation*}
|\gamma_{0}(t)|\leq (|x^{*}| + c_{f}\delta^{*}(x))e^{c_{f}\delta^{*}(x)}=\la(|x|), \quad \forall\ t \in [0,\delta^{*}].
\end{equation*}
Thus, we get
\begin{align*}
	\int_{0}^{\delta^{*}(x)}{L(\gamma_{0}(s), u_{0}(s))\ ds} \leq \delta^{*}(x)\beta(\la(|x|)) \leq D(|x|)\beta(\la(|x|)).
\end{align*}
Hence, we obtain 
 \begin{equation}\label{eq:LAabove1}
 \mathcal{F}_{\varphi}(\gamma,u) \leq \varphi(x^{*}) + D(|x|)\beta(\la(|x|))
 \end{equation}
 which completes the proof of (1) for $(t,x) \in [D(|x|), \infty) \times \R^{d}$. 
 
 We now consider the case $(t,x) \in [0,D(|x|)) \times \R^{d}$. Let $(\gamma_{x}, u_{x}) \in \Gamma_{0,t}^{\to x}$ be defined as
 \begin{equation*}
 u_{x}(s)=0, \quad \gamma_{x}(s) \equiv x, \quad s \in [0,D(|x|)).	
 \end{equation*}
 Then
 \begin{align}\label{eq:LAabove2}
 \mathcal{F}_{\varphi}(x,0) \leq\ \varphi(x) + tL(x,0) \leq \varphi(x) + D(|x|) \beta(|x|).
 \end{align} 
This completes the proof of (1). 


	We now proceed with the proof of (2). We begin by deriving a lower bound for $\mathcal{F}_{\varphi}(\gamma,u)$, for any $(\gamma,u) \in \Gamma_{0,t}^{\to x}$ satisfying \eqref{eq:sublevel}.
	Similarly to the proof of (1) we analyze two cases: first, we show that the conclusion holds for any $(t,x) \in [D(|x|), \infty) \times \R^{d}$ and then we do the same for  $(t,x) \in [0,D(|x|)) \times \R^{d}$. Let $(t,x) \in [D(|x|), \infty) \times \R^{d}$ and let $(\gamma, u) \in \Gamma_{0,t}^{\to x}$ satisfy \eqref{eq:sublevel}. Then, by \eqref{eq:L0} we have that 
	\begin{align}\label{eq:LAbelow}
	\begin{split}
	 \mathcal{F}_{\varphi}(\gamma,u) \geq\ & \varphi(\gamma(0)) + \frac{1}{2\ell_{1}}\int_{0}^{t}{|u(s)|^{2}\ ds} \geq\ \frac{1}{2\ell_{1}}d_{\SR}(x,\gamma(0))^{2} .
	 \end{split}
	\end{align}
Therefore, combining \eqref{eq:LAbelow} with \eqref{eq:LAabove1} we have that 
\begin{align}\label{eq:cphi1}
\frac{1}{2\ell_{1}}d_{\SR}(x,\gamma(0))^{2} \leq\ &	D(|x|)\beta(|x|) + \varphi(x^{*}) 
 \end{align}
which implies \eqref{eq:SRdist} for $(t,x) \in [D(|x|), \infty) \times \R^{d}$ by the continuity of $\varphi$. 

Now, let $(t,x) \in [0,D(|x|)) \times \R^{d}$ and observe that inequality  \eqref{eq:LAbelow} still holds true. So, we combine such an estimate with \eqref{eq:LAabove2} to obtain
\begin{align}\label{eq:cphi2}
\frac{1}{2\ell_{1}}d_{\SR}(x,\gamma(0))^{2} \leq\ & D(|x|)\beta(|x|) + \varphi(x)
\end{align}
which implies \eqref{eq:SRdist} for any $(t,x) \in [0,D(|x|)) \times \R^{d}$, again, by the continuity of $\varphi$. Hence, from \eqref{eq:cphi1} and \eqref{eq:cphi2} we get, for any $R \geq 0$, any $(t, x) \in [0,\infty) \times \overline{B}_{R}$ and any $(\gamma, u) \in \Gamma_{0,t}^{\to x}$ satisfying \eqref{eq:sublevel},
\begin{equation*}
d_{\SR}(x, \gamma(0)) \leq C_{\varphi}(R):= \beta(R)D(R) + \max_{x \in \overline{B}_{R}} \varphi(x).
\end{equation*}

Now, given $(t,x) \in [0,\infty) \times \R^{d}$,  let $\varphi \in \mathcal{S}$ and let $(\gamma, u) \in \Gamma_{0,t}^{\to x}$ be optimal for the minimization problem in \eqref{eq:LAsemi}. Then, we have that 
\begin{align*}
T_{t}\varphi(x) \geq\ & \varphi(\gamma(0)) + \frac{1}{2\ell_{1}}\int_{0}^{t}{|u(s)|^{2}\ ds} \geq\ 0.
\end{align*}
This completes the proof of ($3$). Then, (4) is a direct consequence of the definition of $T_{t}\varphi$.

As for (5), we already know that for any $(t,x) \in [0,\infty) \times \R^{d}$ and any $\varphi \in \mathcal{S}$ we have that $T_{t}\varphi(x) \geq 0$. The proof of the fact that  $T_{t}\varphi(x) \prec L$ and the semigroup property is similar to the proof of \cite[(1) of Proposition 3.3]{bib:FM} and \cite[Proposition 4.6.2.]{bib:FA} and will be omitted here.

We finally show (6). Let $R \geq 0$, let $x \in \overline{B}_{R}$ and let $t \geq 0$. Let $\{\varphi_{n}\}_{n \in \N} \in \mathcal{S}$ and let $\varphi 	\in \mathcal{S}$ be such that $\varphi_{n} \to \varphi$ locally uniformly. Then, on the one hand, taking $(\gamma^{\varphi}_{x}, u^{\varphi}_{x}) \in \Gamma_{0,t}^{\to x}$ optimal for the infimum defining $T_{t}\varphi(x)$ we the  obtain 
\begin{equation*}
T_{t}\varphi_{n}(x)- T_{t}\varphi(x) \leq \varphi_{n}(\gamma^{\varphi}_{x}(0)) - \varphi(\gamma^{\varphi}_{x}(0)).	
\end{equation*}
Hence, from (2) we deduce that there exists a constant $C_{\varphi}(R) \geq 0$ such that 
\begin{equation}\label{eq:cont1}
	T_{t}\varphi_{n}(x)- T_{t}\varphi(x) \leq \|\varphi_{n}(\cdot) - \varphi(\cdot)\|_{\infty, \overline{B}_{C_{\varphi}(R)}}.
\end{equation}
On the other hand, let $(\gamma^{n}_{x}, u^{n}_{x}) \in \Gamma_{0,t}^{\to x}$ optimal for $T_{t}\varphi_{n}(x)$. Then, from (2) there exists a constant $C_{\varphi_{n}}(R) \geq 0$ such that 
\begin{equation}\label{eq:cont2}
	T_{t}\varphi(x)- T_{t}\varphi_{n}(x) \leq \varphi(\gamma^{n}_{x}(0)) - \varphi_{n}(\gamma^{n}_{x}(0)) \leq \|\varphi_{n}(\cdot) - \varphi(\cdot)\|_{\infty, \overline{B}_{C_{\varphi_{n}}(R)}}.
\end{equation}
Moreover, owing to the locally uniform convergence of the sequence $\varphi_{n}$ the constant $C_{\varphi_{n}}(R)$ is independent of $n \in \N$. Therefore,  the proof of (6) is completed by combining \eqref{eq:cont1} and \eqref{eq:cont2}. \qed

We call $T_{t}$ the Lax-Oleinik semigroup, adapted to sub-Riemannian systems. In order to find a critical solution $\overline\chi$ such that
 \[\overline\chi(x)=T_{t}\overline\chi(x), \quad \forall\ t \geq 0,\,\, \forall\ x \in \R^{d}\] 
we will take  the critical solution $\chi$ in $\mathcal{S}$ given by \Cref{thm:convergence} and show that  $T_{t}\chi(x)$ converges as the $t \to \infty$  to a function $\overline\chi$. This will be accomplished in two steps, which are the object of our next proposition and of the following theorem. 


\begin{proposition}\label{prop:LAequi}
	Assume {\bf (F1)}, {\bf (F2)} and  {\bf (L0)} -- {\bf (L3')}. Then, for any $R \geq 0$ we have that
	\begin{itemize}
		\item[($i$)] $\{T_{t}\chi\}_{t \geq 0}$ is equibounded on $\overline{B}_{R}$;
		\item[($ii$)] $\{T_{t}\chi\}_{t \geq 1}$ is equicontinuous on $\overline{B}_{R}$.	
		\end{itemize}
\end{proposition}


\proof 

In order to prove ($i$) we argue as in \Cref{thm:SRbounds}. Let $R \geq 0$, let $t \geq 0$ let $x \in \overline{B}_{R}$. Let $(\gamma_{x}, u_{x}) \in \Gamma_{0,t}^{\to x}$ be optimal for $T_{t}\chi(x)$. Since $\chi$ is Lipschitz continuous w.r.t. $d_{\SR}$ the following holds
\begin{equation}\label{eq:Chibounds1}
\chi(x) \leq \ell_{R} d_{\SR}(x,0), \quad \forall\ x \in \overline{B}_{R}.
\end{equation}
 Then, since $\chi \geq 0$ from \eqref{eq:L0} we obtain 
\begin{align}\label{eq:bound1}
\begin{split}
 T_{t}\chi(x) \geq\   \chi(\gamma(0)) + \frac{1}{2\ell_{1}}\int_{0}^{t}{|u_{x}(s)|^{2}\ ds}  =\ \frac{1}{2\ell_{1}}d_{\SR}(x,\gamma_{x}(0))^{2}
 \end{split}
\end{align}
which is bounded by (2) in \Cref{thm:SRbounds}.
Thus, if $t \in [D(R), \infty)$, then by \eqref{eq:LAabove1} and \eqref{eq:Chibounds1} we obtain 
\begin{align}\label{eq:bound2}
\begin{split}
 T_{t}\chi(x) \leq \chi(x^{*}) + D(R)\beta(R) \leq\ \ell_{R_{\K}}d_{\SR}(x^{*},0)  + D(R)\beta(R)	
\end{split}
\end{align}
where $R_{\K}$ stands for the diameter of $\K$. 
On the other hand, if $t \in [0,D(R))$, by \eqref{eq:LAabove2} and \eqref{eq:Chibounds1} we get
\begin{align}\label{eq:bound3}
\begin{split}
 T_{t}\chi(x)  \leq \chi(x) + D(R)\beta(R) \leq\ \ell_{R}d_{\SR}(x,0)  + D(R)\beta(R).	
\end{split}
\end{align}
Hence, combining \eqref{eq:bound1} with \eqref{eq:bound2} and, also, \eqref{eq:bound1} with \eqref{eq:bound3} the proof of ($i$) is complete. 

We now proceed to show ($ii$), that is, the equicontinuity of $T_{t}\chi(x) $ for $t \geq 1$. Let $R \geq 0$, let $x$, $y \in \overline{B}_{R}$ and let $t \geq 1$. To begin with, assume that $d_{\SR}(x,y) > 1$. Then, we have that 
\begin{align*}
|T_{t}\chi(x)-T_{t}\chi(y)| \leq 2 \| T_{t}\chi\|_{\infty, \overline{B}_{R}} 	\leq 2\| T_{t}\chi\|_{\infty, \overline{B}_{R}} d_{\SR}(x,y).
\end{align*}

Next, suppose $d_{\SR}(x,y) \leq 1$ so that $d_{\SR}(x,y) \leq t$. Let $(\gamma_{0}, u_{0}) \in \Gamma_{0, d_{\SR}(x,y)}^{y \to x}$ be optimal for \eqref{eq:subriem} and let $(\gamma_{y}, u_{y}) \in \Gamma_{0,t}^{\to y}$ be optimal for $T_{t}\chi(y)$. Define the control 
\begin{equation*}
\widetilde{u}(s)=
\begin{cases}
	u_{y}(s+d_{\SR}(x,y)), & s \in [0,t-d_{\SR}(x,y)]
	\\
	u_{0}(s-t+d_{\SR}(x,y)), & s \in (t-d_{\SR}(x,y), t]
\end{cases}	
\end{equation*}
and call $\widetilde\gamma$ the corresponding trajectory, that is, $(\widetilde\gamma, \widetilde{u}) \in \Gamma_{0,t}^{\to x}$. Note that $\widetilde{u}$ can be used to estimate $T_{t}\chi(x)$ from above. We have that 
\begin{align}\label{eq:equicont}
\begin{split}
& T_{t}\chi(x)-T_{t}\chi(y) 
\\
\leq\ & \chi(\widetilde\gamma(0)) - \chi(\gamma_{y}(0)) + \int_{0}^{t}{L(\widetilde\gamma(s), \widetilde{u}(s))\ ds} - \int_{0}^{t}{L(\gamma_{y}(s), u_{y}(s))\ ds}
\\
=\ & \chi(\gamma_{y}(d_{\SR}(x,y))) - \chi(\gamma_{y}(0))
+ \int_{0}^{t}{L(\gamma_{y}(s), u_{y}(s))\ ds} 
\\
-\ & \int_{0}^{d_{\SR}(x,y)}{L(\gamma_{y}(s), u_{y}(s))\ ds} + \int_{0}^{d_{\SR}(x,y)}{L(\gamma_{0}(s), u_{0}(s))\ ds} - \int_{0}^{t}{L(\gamma_{y}(s), u_{y}(s))\ ds}
\\
=\ & \chi(\gamma_{y}(d_{\SR}(x,y))) - \chi(\gamma_{y}(0)) - \int_{0}^{d_{\SR}(x,y)}{L(\gamma_{y}(s), u_{y}(s))\ ds} + \int_{0}^{d_{\SR}(x,y)}{L(\gamma_{0}(s), u_{0}(s))\ ds}.
\end{split}
\end{align}
We estimate first the integral terms. By \eqref{eq:L0} we immediately obtain 
\begin{equation*}
\int_{0}^{d_{\SR}(x,y)}{L(\gamma_{y}(s), u_{y}(s))\ ds} \geq 0.
\end{equation*}
Moreover, since $\|u_{0}\|_{\infty, [0, d_{\SR}(x,y)]} \leq 1$ we have that 
\begin{equation*}
|\gamma_{0}(t)|\leq (|x| + c_{f}d_{\SR}(x,y))e^{c_{f}d_{\SR}(x,y)}, \quad \forall\ t \in [0,d_{\SR}(x,y)].
\end{equation*}
So, we get
\begin{equation*}
\int_{0}^{d_{\SR}(x,y)}{L(\gamma_{0}(s), u_{0}(s))\ ds} \leq d_{\SR}(x,y)\big(|x| + c_{f}d_{\SR}(x,y)\big)e^{c_{f}d_{\SR}(x,y)}.
\end{equation*}
Combining both inequalities we have that 
\begin{align}\label{eq:equi1}
\begin{split}
	 -\ & \int_{0}^{d_{\SR}(x,y)}{L(\gamma_{y}(s), u_{y}(s))\  ds} +\  \int_{0}^{d_{\SR}(x,y)}{L(\gamma_{0}(s), u_{0}(s))\ ds} 
	 \\
	 \leq\ &  d_{\SR}(x,y)\big(|x| + c_{f}d_{\SR}(x,y)\big)e^{c_{f}d_{\SR}(x,y)}.
	\end{split}
\end{align}
Therefore, in order to complete the proof we need to estimate 
\begin{equation}\label{eq:oscillation}
\chi(\gamma_{y}(d_{\SR}(x,y))) - \chi(\gamma_{y}(0)). 
\end{equation}
First, we claim that $|\gamma_{y}(0)|$ and $|\gamma_{y}(d_{\SR}(x,y))|$ are bounded. Indeed, observe that from (2) in \Cref{thm:SRbounds} and the equivalence of the sub-Riemannian topology with the Euclidean one we deduce that 
\begin{equation*}
|\gamma_{y}(0)| \leq 2\max\{R, C_{\chi}(R)\}. 	
\end{equation*}
Moreover, by \Cref{lem:L2bounds} we know that 
\begin{equation*}
|\gamma_{y}(s)| \leq \kappa(\|u_{y}\|_{2},1)(1+|\gamma_{y}(0)|), \quad \forall\ s \in [0,d_{\SR}(x,y)].	
\end{equation*}
So, in particular, 
\begin{equation*}
	|\gamma_{y}(d_{\SR}(x,y))| \leq\ \kappa(\|u_{y}\|_{2},1)(1+|\gamma_{y}(0)|).
\end{equation*}
We claim that $\|u_{y}\|_{2, [0,d_{\SR}(x,y)]}$ is bounded by a constant that only depends on $R$. Indeed, from ($i$) we know that $T_{t}\chi(y)$ is locally uniformly bounded and by \eqref{eq:L0} we know that 
\begin{align*}
T_{t}\chi(y)  \geq\  \chi(\gamma_{y}(0)) + \frac{1}{2\ell_{1}}\int_{0}^{t}{|u_{y}(s)|^{2}\ ds} \geq\  \frac{1}{2\ell_{1}}\int_{0}^{t}{|u_{y}(s)|^{2}\ ds}. 
\end{align*}
Thus, we obtain
\begin{equation*}
	\frac{1}{2\ell_{1}}\int_{0}^{t}{|u_{y}(s)|^{2}\ ds} \leq \|T_{t}\chi(y)\|_{\infty, \overline{B}_{R}}
\end{equation*}
and this proves the claim since $T_{t}\chi$ is locally equibounded by ($i$). For simplicity of notation, let $R_{y} \geq 0$ be such that
\begin{equation*}
|\gamma_{y}(d_{\SR}(x,y))| \leq R_{y}, \quad |\gamma_{y}(0)|\leq R_{y},
\end{equation*}
and denote $r_{y} \geq 1$ the degree of nonholonomy associated with the compact set $\overline{B}_{R_{y}}$. 

Since $\chi$ is Lipschitz continuous w.r.t. $d_{\SR}$ we get
\begin{equation*}
	\chi(\gamma_{y}(d_{\SR}(x,y))) - \chi(\gamma_{y}(0)) \leq \ell_{R_{y}}d_{\SR}(\gamma_{y}(d_{\SR}(x,y)), \gamma_{y}(0)).
\end{equation*}
Then, by \Cref{cor:ballbox} we have that 
\begin{align}\label{eq:oscillation1}
	\chi(\gamma_{y}(d_{\SR}(x,y))) - \chi(\gamma_{y}(0)) \leq \widetilde{c}_{2}|\gamma_{y}(d_{\SR}(x,y))- \gamma_{y}(0)|^{\frac{1}{r_{y}}}
\end{align}
where $\widetilde{c}_{2}$ depends only on $R_{y}$. Moreover, from \Cref{lem:L2bounds}  we have that 
\begin{equation*}
|\gamma_{y}(d_{\SR}(x,y))-\gamma_{y}(0)| \leq \kappa\big(\|u_{y}\|_{2, [0,d_{\SR}(x,y)]}, d_{\SR}(x,y)\big)(1+|\gamma_{y}(0)|)d_{\SR}(x,y)^{\frac{1}{2}}.
\end{equation*}
Hence, we conclude that there exists a constant $C^{\prime}_{R} \geq 0$ such that 
\begin{equation}\label{eq:oscillation}
\chi(\gamma_{y}(d_{\SR}(x,y))-\chi(\gamma_{y}(0)) \leq C^{\prime}_{R}d_{\SR}(x,y)^{\frac{1}{2r_{y}}}.	
\end{equation}
Therefore, combining \eqref{eq:equicont}, \eqref{eq:equi1} and \eqref{eq:oscillation} we obtain
\begin{equation*}
T_{t}\chi(x)-T_{t}\chi(y) \leq\ d_{\SR}(x,y)\big(|x| + c_{f}d_{\SR}(x,y)\big)e^{c_{f}d_{\SR}(x,y)} +  C^{\prime}_{R}d_{\SR}(x,y)^{\frac{1}{2r_{y}}}.
\end{equation*}
The proof can be completed by exchanging the role of $x$ and $y$.\qed

\begin{theorem}\label{thm:Lax}
	Assume {\bf (F1)}, {\bf (F2)} and  {\bf (L0)} -- {\bf (L3')}. Then, there exists a continuous function $\overline\chi$ such that 
	\begin{equation}\label{eq:barlimit}
	\lim_{t \to \infty} T_{t}\chi(x) = \overline\chi(x)
	\end{equation}
uniformly on $\overline{B}_{R}$ for any $R \geq 0$. Moreover, we have that 
\begin{equation*}
\overline\chi(x)=T_{t}\overline\chi(x) \quad (t \geq 0,\,\ x \in \R^{d})
\end{equation*}
and $\overline\chi$ satisfies 
\begin{equation}\label{eq:HJcri}
H(x, D\overline\chi(x))=0 \quad (x \in \R^{d})	
\end{equation}
in the viscosity sense.
\end{theorem}
\proof
In order to prove the existence of the limit in \eqref{eq:barlimit}, we first show that the map 
\begin{equation*}
t \mapsto T_{t}\chi(x)
\end{equation*}
is nondecreasing for any $x \in \R^{d}$. Indeed, we have that for any $s$, $t \geq 0$ 
\begin{equation*}
T_{t}\chi(x) \leq T_{t}\left(T_{s}\chi(x)\right)= T_{t+s}\chi(x)
\end{equation*}
where the inequality holds since $\chi \prec L $. This implies that 
 \begin{equation*}
 T_{t}\chi(x)  \leq T_{t^{\prime}}\chi(x), \quad \forall\ t \leq t^{\prime}.	
 \end{equation*}
Therefore, since $T_{t}\chi$ is locally equibounded by \Cref{prop:LAequi} the  limit 
\begin{equation*}
\overline\chi(x):=\lim_{t \to \infty} T_{t}\chi(x)
\end{equation*}
exists for all $x \in \R^{d}$. Moreover, again by \Cref{prop:LAequi} we know that the family $T_{t}\chi$ is locally equicontinuous. Thus the above limit is locally uniform.

Next, in order to show that $\overline\chi(x)=T_{t}\overline\chi(x)$ for any $x \in \R^{d}$ and any $t \geq 0$, let $s \geq 0$. Then
\begin{align*}
	T_{s}\overline\chi(x) = \lim_{t \to \infty} T_{s}\big(T_{t}\chi(x)\big)=\lim_{t \to \infty} T_{s+t}\chi(x)
\end{align*}
where we have used the continuity of $T_{t}$ and property (4) in \Cref{thm:SRbounds}.  
Hence, 
\begin{equation*}
T_{s}\overline\chi(x) = \lim_{t \to \infty} T_{s+t}\chi(x)= \overline\chi(x), \quad \forall\ s \geq 0
\end{equation*}
So, we have that
\begin{align}\label{eq:fixed}
\overline\chi(x)=\   \inf_{(\gamma,u) \in \Gamma_{0,t}^{\to x}}\left\{\overline\chi(x) + \int_{0}^{t}{L(\gamma(s), u(s))\ ds} \right\} \quad \forall\ t \geq 0.
\end{align}
The proof of the fact that the function $\overline\chi$ solves \eqref{eq:HJcri} in the viscosity sense is similar to the proof of \cite[Proposition 5.1, Proposition 5.2]{bib:FM}. \qed

\appendix
\section{Abelian-Tauberian Theorem}
\label{sec:Appendix1}
In this appendix, we give a new formulation of the Abelian-Tauberian Theorem, stated in \cite[Theorem 5]{bib:MA}, tailored for the proof of \Cref{thm:convergence}. 

\begin{theorem}\label{thm:newabelian}
	Let $\psi(t,x)$ be the solution of
	\begin{align*}
	\begin{cases}
		\partial_{t} \psi(t,x) + H(x, D\psi(t,x))=0, & \quad (t,x) \in [0,T] \times \R^{d}
		\\
		\psi(T,x)=0, & \quad x \in \R^{d}.
	\end{cases}
	\end{align*}
For any $\lambda > 0$, let $\psi_{\lambda}(x)$ be the solution of
\begin{equation*}
\lambda \psi(x) + H(x, D\psi(x))=0, \quad x \in \R^{d}. 	
\end{equation*}
Then:
\begin{itemize}
\item[($i$)] if $\{\lambda \psi_{\lambda}(\cdot\ )\}_{\lambda > 0}$ locally uniformly converges to a constant $\bar{d} \in \R$ as $\lambda \downarrow 0$, then $\{\frac{1}{T}\psi(0,\cdot\ )\}_{T > 0}$ locally uniformly converges to $\bar{d}$ as $T \to \infty$;
\item[($ii$)] if $\{\frac{1}{T}\psi(0,\cdot\ )\}_{T > 0}$ locally uniformly converges to a constant $\bar{d} \in \R$ as $T \to \infty$, then $\{\lambda \psi_{\lambda}(\cdot)\}_{\lambda > 0}$ locally uniformly converges to $\bar{d}$ as $\lambda \downarrow 0$.
\end{itemize}
\end{theorem}

This result can be proved arguing as in \cite[Theorem 5]{bib:MA} keeping in mind the following differences:
\begin{enumerate}
	\item the uniform convergence on the full space $\overline\Omega$ is replaced by the locally uniform convergence on $\R^{d}$;
	\item whenever the boundedness assumption on $L$ is used in \cite{bib:MA} one here has to invoke the boundedness of optimal pairs $(\gamma, u)$ in $L^{\infty}(0,T; \R^{d}) \times L^{2}(0,T; \R^{m})$. 
\end{enumerate}

\bibliographystyle{abbrv}
\bibliography{references}

\end{document}